\documentclass[reqno,centertags,12pt]{amsart}
\usepackage{amsmath,amsthm,amscd,amssymb,latexsym,verbatim}
\usepackage{amssymb}

\usepackage{mathtools}

\usepackage{graphicx,epsf,cite}

\usepackage[shortlabels]{enumitem}

\usepackage{color}

-\marginparwidth \oddsidemargin -4mm \evensidemargin \oddsidemargin

\usepackage{setspace}
\newtheorem{theorem}{Theorem}[section]

\newtheorem{proposition}[theorem]{Proposition}
\newtheorem{lemma}[theorem]{Lemma}
\newtheorem{corollary}[theorem]{Corollary}
\theoremstyle{definition}
\newtheorem{definition}[theorem]{Definition}

\newcounter{smalllist}

\DeclareMathOperator*{\dist}{dist} 

\DeclareMathOperator*{\Lip}{Lip}

\allowdisplaybreaks
\numberwithin{equation}{section}

\newcommand{\lb}{\label}

\newcommand{\wu}{\widetilde{\mu}  }

\newcommand{\beq}{\begin{equation}}
\newcommand{\eeq}{\end{equation}}

\newcommand{\bal}{\begin{align}}
\newcommand{\eal}{\end{align}}
\newcommand{\bals}{\begin{align*}}
\newcommand{\eals}{\end{align*}}

\newcommand{\bbR}{{\mathbb{R}}}

\newcommand{\bbT}{{\mathbb{T}}}
\newcommand{\bbS}{{\mathbb{S}}}

\newcommand{\calQ}{{\mathcal Q}}

\newcommand{\calG}{{\mathcal G}}

\newcommand{\eps}{\varepsilon}

\usepackage[T1]{fontenc}
\usepackage{esint}

\begin{document}
\title[Regularity of Hele-Shaw Flow]
{$C^{1}$-Regularity of the Free Boundary for Hele-Shaw Flow with Source and Drift}

\author{Yuming Paul Zhang}
\address{\noindent Department of Mathematics and Statistics \\ Auburn University \\ Parker Hall, 221 Roosevelt Concourse, Auburn, AL 36849 \newline Email: \tt
yzhangpaul@auburn.edu}



\begin{abstract} 
This paper is a continuation of the work in \cite{kimzhang2024} concerning Hele-Shaw flow with both drift and source terms. We prove that, in a local neighborhood, if the free boundary is Lipschitz continuous with a sufficiently small Lipschitz constant, then the free boundary is $C^{1}$. As a corollary, we also consider the 2D vertical Hele-Shaw (or one-phase Muskat) problem with an advection term. We show that, provided the initial data and the advection term are small and the propagation speed is  large, the free boundary becomes uniformly $C^1$ after a finite time.
\end{abstract}

\maketitle

\medskip


\medskip


\section{Introduction} 

This paper provides a continuation of the study of Hele-Shaw flow with source and drift terms in \cite{kimzhang2024}. 
We consider the problem where $u$ solves the following system within the positive set $\{(x,t) \in B_{2}\times (-2,2)\, |\, u(x,t) > 0\}$:
\begin{equation}\lb{1.1}
    \left\{\begin{aligned}
        -\Delta u &=f \quad &&\text{ in }\{u(x,t)>0\},\\
        u_t&=|\nabla u|^2+{b}\cdot\nabla u\quad &&\text{ on } \partial\{u(x,t)>0\},
    \end{aligned}\right.
\end{equation}
Here, $b$ represents a Lipschitz vector field  and $f$ is a non-negative H\"older continuous function.  The set $\partial\{u(x,t)>0\}$ is referred to as the free boundary.
The presence of source and drift terms is not just a mathematical extension but is also deeply rooted in practical applications, including tumor growth \cite{PQV,David_S,jacobs2022tumor} and population dynamics \cite{maury2010,CKY}. 

\smallskip

It was proved in \cite{kimzhang2024} that if the free boundary of the solution is locally close to a Lipschitz graph with small Lipschitz constant, then the free boundary is indeed Lipschitz. Moreover, we have the non-degeneracy property. Our main result demonstrates that if the free boundary is Lipschitz continuous with a sufficiently small Lipschitz constant in a local neighborhood, it must necessarily be $C^1$. 

Compared to the classical model, this formulation introduces lower-order terms in both the governing equation and the free boundary condition. It appears that after zooming in at a single point, $f\to 0$ and ${b}$ becomes a constant vector fields, so the classical theory applies as well. However, realizing this involves serious obstacles. Indeed, the primary challenge in this analysis stems from the irregular nature of the free boundary. Because the boundary is not prescribed, its shape can be rough, causing the normal directions and the solution itself to be highly sensitive to small perturbations from the extra drift terms. 

\smallskip

Now, we state our main result. In it, we denote 
\[
\calQ_{r,c}:=B_r\times (-c,c)\quad\text{and}\quad \calQ_{r}:=\calQ_{r,r}.
\]

\begin{theorem}{\rm[$C^1$ free boundary]}\lb{T.8.4}
Suppose that $u$ is non-decreasing along all directions in $W_{\theta,-e_d}$ with $\theta\geq \theta_{*,d}$ for some dimensional constant $\theta_{*,d}$ in $\calQ_2$ and there exist $C>1$ and $c>0$ such that
\begin{equation}\lb{nondeg}
     u(x-\eps e_d,t)\geq \eps/C \quad \hbox{ for all } (x,t)\in \Gamma_u\cap \calQ_2\text{ and all }\eps>0,
\end{equation}
\beq\lb{cond*}
| \partial_t u(-e_d,t)|\leq C,\quad  u(-e_d,t)\in [C^{-1},C] \quad \forall t\in (-2c,2c).
\eeq
Then the free boundary $\Gamma_u\cap \calQ_{1,c}$  is $C^1$.

Moreover, if $\Gamma_u(t)$ for some $t\in (-c,c)$ is given by a graph $F(x')=x_d$ with $(x',x_d)\in\bbR^d$ locally near one free boundary point $(0,x_{0,d})\in \Gamma_u(t)$, then $F$ is differentiable and there exist $K,C'\geq 2$ such that for $x',y'$ sufficiently close to $0$,
\[
|\nabla F(x')-\nabla F(y')|\leq C' (-\log_K |x'-y'|)^{-1}.
\]
\end{theorem}

The condition \eqref{nondeg} is often referred to as the non-degeneracy condition. 
Since our result is local in nature, the condition \eqref{cond*} imposed on $u$ at the point $-e_d$ is essential to prevent sudden change originating from outside the local neighborhood. Furthermore, the assumptions of monotonicity and non-degeneracy used here are consistent with the results of \cite[Theorems AB and 7.5]{kimzhang2024}, which established these properties under even broader conditions.

In the specialized case where the drift vanishes ($b \equiv 0$), \cite[Corollary 6.3]{kimzhang2024} demonstrates that the free boundary achieves $C^{1,\alpha}$ regularity. This is achieved by transforming the equation into an obstacle problem for each $t>0$ and applying the regularity estimates from \cite[Theorem 7.1]{blank2001sharp}. However, in the presence of a general drift field $b$, we are only able to establish $C^1$ regularity. This limitation arises from the inherent difficulty in controlling the variation of the solution along the temporal direction as well as the fluctuations coming from the lower order terms. The regularity obtained here is the same as \cite{kim2, CJK} for the classical problem.

We postpone the discussion of the new ingredients in the proof to Subsection \ref{s.1.3}.

\subsection{A 2D advection Muskat problem}

The Hele-Shaw problem is closely related to the Muskat problem (or Hele-Shaw problem with gravity or vertical Hele-Shaw), to which it can be formally reformulated. We will also investigate a 2D advection Muskat problem.  
Consider the periodic domain $(x,y,t)\in \bbT\times \bbR\times [0,\infty)$, and set the vector field $b=(0,\Theta(x))$ for some periodic Lipschitz continuous function $\Theta(x)$.

The goal is to study the propagation property of front-like solutions. So, we look for solutions $u$ such that
$u(x,y,t)\asymp-c_*y$  as $y\to -\infty$ for some $c_*$,
and the free boundary $\Gamma_u(t)$ is represented by a graph $y=F(x,t)$. 
In the case where $\Theta\equiv 0$, $u=c_*(c_*t-y)_+$ solves \eqref{2.6} and the free boundary is given by the graph $y=c_*t$.  When $\Theta\equiv c_*$, $u=c_*(-y)_+$ is a solution with free boundary $F\equiv 0$.

The equation of the graph $F$ is given by
\beq\lb{2.6'}
\partial_t F(x,t)=-{c_*}\calG (F(\cdot,t))(x)+c_*-\Theta (x)\quad\text{ in }\bbT\times[0,\infty),
\eeq
with initial data $F_0$. Here $\calG$ is the Dirichlet-to-Neumann operator defined in \eqref{calG}.

When $\Theta\equiv c_*$, the initial value problem of \eqref{2.6'} (also known as the Muskat problem) was studied in \cite{dong21,dong23,schwab2024well}. They showed the global wellposedness of solutions and that the free boundary is uniformly Lipschitz continuous for all time by establishing the comparison principle for the equation of the free boundary. 
Compared with the aforementioned literatures, our results are different in the following aspects: we consider general advection equations; we show not only uniform Lipschitz continuity of the free boundary but also that the free boundary becomes uniformly $C^1$ after a finite time.

\begin{theorem}{\rm[Advection Muskat problem]}\lb{T.main.2}
Let $F_0,\Theta:\mathbb{T}\to\bbR$ be Lipschitz continuous. Then
there exists a unique viscosity solution $F$ to \eqref{2.6'} with initial data $F_0$, and $F$ is Lipschitz continuous in space locally uniformly in time.

Suppose that the Lipschitz constant of $F_0$ is strictly smaller than  $1/{\sqrt{3}}$ and $\sup_{x\in\bbT}\calG(F_0)<1$. Then there exists $C\geq 1$  depending only on $F_0$ and $\|\Theta\|_{\infty}$ such that if $c_*\geq C$ and $\|\Theta\|_{\Lip}\leq 1/C$, we have that $F$ is uniformly Lipschitz continuous for all $t\geq 0$ and is uniformly  $C^1$ continuous in space for all $t\geq T_0$ for some $T_0>0$.
\end{theorem}

When $F_0$ and $\Theta$ are sufficiently flat and $c_*$ is sufficiently large, it can be shown that the free boundary $F$ is strictly advancing at a finite speed and the solution $u$ is non-decreasing along streamlines (Lemma \ref{L.2.6}). Estimates on $\calG(F)$ in terms of $F$ are provided in \cite[Propositions 2.20, 2.21]{de2017paradifferential}. We also assume an initial Lipschitz constant of less than $1/\sqrt{3}$ to establish a uniform non-degeneracy property of the free boundary. It should be noted, however, that these conditions are by no means optimal.



The global well-posedness is established by the comparison principle and the perron method for the equation of the free boundary $F$.
However, unlike the aforementioned papers, with the appearance of the drift term, the comparison principle only gives local Lipschitz continuity which blows-up as time tends to infinity.

To overcome this difficulty, we leverage the smoothing mechanism provided by the system in Theorem~\ref{T.8.4} to improve the regularity of the free boundary and establish uniform $C^1$ estimates. First, we show that the two notions of viscosity solution — one defined through the free boundary equation and the other through the equation for $u$ — are equivalent, provided either that $F_t$ is bounded from above or that the solution is non-degenerate (Lemma~\ref{L.5.2}). We then show that the graph of the free boundary is uniformly Lipschitz continuous (which is slightly stronger than Lipschitz free boundary) and that the Lipschitz constant decreases in time. The interior gain comes from the observation that $\nabla u$ is almost vertical (Lemma~\ref{L.6.1}), and then we apply the sup-convolution technique (Proposition~\ref{L.5.10}).



\subsection{Literature review}

The story begins with the classical Hele-Shaw problem in which $f=b=0$ \cite{HS1898,richardson1972hele}. The equation describes a fluid moving between two narrow, parallel plates. Early mathematical treatments approached it through variational inequalities \cite{EJ} and classical solution theory for smooth initial data \cite{escher1997classical}, while \cite{constantin1993global} established global existence for small analytic perturbations.


\smallskip

The ``Golden Era'' of this theory was defined by the landmark works of Caffarelli and others, who developed several powerful PDE tools to study the regularity properties of free boundaries (see \cite{cbook}). Building upon these foundations, a series of papers by Kim provided a thorough study of the Hele-Shaw problem. This began with the introduction of viscosity solutions, establishing the existence and uniqueness of solutions in \cite{kim2003}. Subsequently, it was proven in \cite{kim2} that the free boundary is $C^1$ provided it is Lipschitz continuous and non-degenerate. Later, \cite{CJK, CJK2} showed that the free boundary becomes smooth immediately, even when initiated from a rough state. The regularity of free boundaries at large time was addressed in \cite{kim3}. Further advancements include the work in \cite{figalli20}, which leveraged the problem’s connection to the obstacle problem to establish generic regularity for free boundaries. In contrast, researchers have also explored the limits of the smoothing effect; for instance, \cite{persis, JK1} demonstrated that singularities on the free boundary can persist for a short period before regularity is achieved.


\smallskip
The classical Hele-Shaw problem is equivalent to the one-phase Muskat problem, which was first introduced by \cite{muskat} to describe the evolution of the interface of fluids in a porous medium governed by Darcy's law. 
A large number of works address well-posedness by reducing to the equation for the graph of the boundary, including \cite{ala21, ams20, dong21, dong23, np20}. Local well-posedness was established in \cite{cheng2016well, cordoba2011interface}, among others. In contrast, \cite{ccfg2013, ccfg2016} showed that splash singularities can form from non-graph initial data even in the stable regime, underscoring the critical role of the graph condition. For small initial data, global well-posedness was obtained in \cite{nguyen2022global} in a scaling-invariant Besov space. \cite{constantin2013global,constantin2016muskat,gancedo2019muskat} concerned ``medium data'' in the Wiener algebra, and \cite{cameron2018global,cameron2020global} in Lipschitz norm. For arbitrarily large Lipschitz initial data, 
global well-posedness and Lipschitz regularity were established in two spatial dimensions in \cite{dong21} and in three spatial dimensions in \cite{dong23}; the well-posedness results were generalized to all dimensions in \cite{schwab2024well} for viscosity solutions, and in \cite{ak23} for variational solutions.

\smallskip

The general form of our main equation can be derived from  biological modeling, particularly the spread of tumor cells. Much of the recent literature has focused on the incompressible limit to  Hele-Shaw type flows from models that use Porous Medium diffusion. Within this framework, the convergence of solutions has been rigorously justified (see e.g., \cite{PQV,kim2023incompressible,david2021free,SulakTuranova}), as has the convergence of the free boundaries themselves (\cite{tong2025convergence}). Furthermore, nutrient interactions were introduced, which necessitates the inclusion of the drift term in the free boundary condition \cite{jacobs2022tumor,David_S}. 

Another major theme in research concerns flows in complex or heterogeneous media, governed by periodic or stochastic oscillations. We only list a few papers in this direction \cite{contact,KimMellet2009,povzar2015homogenization,turanova2025hele} etc.



\subsection{New ingredients in proving $C^1$-regularity}\lb{s.1.3}
As discussed in \cite{kimzhang2024}, the regularization mechanism of this system arises from the parabolicity induced by the $|\nabla u|^2$ term. To improve the free boundary regularity from Lipschitz to $C^1$, we adopt the ``cone of monotonicity'' strategy developed in \cite{cbook} and applied to Hele-Shaw in \cite{kim2, CJK}. 

Unlike the parabolic cases in \cite{cbook}, the Hele-Shaw problem requires additional effort to obtain temporal regularity. We address this by applying the $u_t$ control developed in \cite{kim2, CJK}. For problems with lower-order terms, we combine this approach with \cite[Proposition 4.1]{kimzhang2024}, which establishes that harmonic and superharmonic functions remain comparable near the free boundary in domains with small Lipschitz constants.


\smallskip

The first step in Caffarelli's strategy is the interior improvement of monotonicity. Consider a solution $u > 0$ in $B_2(0)$ with a primary growth direction $\nu = \frac{\nabla u(0)}{|\nabla u(0)|}$. 
Caffarelli observed that if $u$ is non-decreasing along a direction $\sigma$ (then $\sigma \cdot \nu \geq 0$), this monotonicity can be improved uniformly across $B_1(0)$. Specifically, there exists a radius $r \geq 0$, proportional to  $\sigma \cdot \nu =  |\nu| \cos\langle\sigma, \nu\rangle$, such that $u$ is monotone along all directions in $\overline{B_r(\sigma)}$ for every point in the unit ball.  This widening of the monotonicity cone in the interior provides the necessary gain to reach the free boundary.

In our setting, we further require this uniform improvement to hold across both space and time while simultaneously compensating for fluctuations introduced by the lower-order terms. Specifically, we must ensure that $\sigma \cdot \nu$ remains strictly positive and demonstrate that monotonicity still improves despite these perturbations. This necessitates a more delicate analysis with a carefully selected positive lower bound. And we establish a new intermediate cone lemma (Lemma \ref{L.cone}) that refines the classical result in \cite{cbook} to account for the influence of source and drift effects.

\smallskip

Another distinct feature of our proof is that, unlike the approaches in \cite{CJK} and \cite{kim2}, we work in a local setting and without having uniform boundedness of $\nabla u$. Enlightened by the idea in \cite{CJK}, we rescale the solution. For a constant $K \geq 1$ to be determined, we define the sequence of rescaled solutions for $n \geq 1$:
$$
u_n(x,t) := \gamma_n K^n u(K^{-n} x, K^{-n} \gamma_n t),
$$
where the scaling factor $\gamma_n$ is defined by:
$$\gamma_n := 
\begin{cases} 
\left[ K^n u(K^{-n} \tfrac{3}{4} \mu_n, 0) \right]^{-1} & \text{if } u(K^{-n} \tfrac{3}{4} \mu_n, 0) \leq K^{-n}, \\ 
1 & \text{otherwise,}
\end{cases}$$
and $\mu_n$ is a unit spatial vector. 
The strategy is to iteratively improve the monotonicity of $u_n$, which yields higher regularity for the free boundary as we zoom into the free boundary. This iteration requires $u_n$ to remain well-behaved over a time interval sufficient for monotonicity gain, implying $\gamma_n$ cannot become arbitrarily small. While $\gamma_n$ is difficult to control independently, Lemma \ref{Lipdomain} shows that $K \gamma_{n+1} \geq 2\gamma_n$ holds uniformly for some $K$, providing the necessary temporal scale for the gain to propagate.


\subsection{Outline}
Section \ref{S2} is devoted to notations, definitions and several preliminary lemmas. In Section \ref{S3}, we prove the interior gain of monotonicity. In Section \ref{S4}, we prove the main result of $C^1$ free boundary. The last Section \ref{S5} discusses the vertical Hele-Shaw problem (or Muskat problem) with advection terms. 

\subsection{Acknowledgements}
Y. P. Zhang acknowledges support from NSF CAREER grant DMS-2440215 and Simons Foundation Travel Support MPS-TSM-00007305. The author is grateful to Inwon Kim for the valuable comments.

\section{Preliminaries}\lb{S2}


Letting $u:\bbR^d\times [0,\infty)\to [0,\infty)$, we write 
\[
\Omega_u:=\{u(\cdot,\cdot)>0\},\quad \Omega_u(t):=\{u(\cdot,t)>0\},
\]
\[
\Gamma_u(t):=\partial \Omega_u(t),\quad \Gamma_u:=\bigcup_t\, \Gamma_u(t)\times\{t\} .
\]
We also denote 
\[
\calQ_{r,s}(x,t):=B_r(x)\times (t-s,t+s).
\]
For $r>0$, we denote
\[
\calQ_r:=B_r\times (-r,r).
\]

The notions of viscosity solution to \eqref{1.1} is given in \cite{kim2003} (also see \cite{kimzhang2024}). 

\begin{definition}\lb{D.21}
A non-negative continuous function $u$ defined in $\calQ_r$ is a viscosity subsolution of \eqref{1.1} if for every $\phi\in C^{2,1}_{x,t}(\calQ_r)$ such that $u-\phi$ has a local maximum in $\overline{\Omega_u}\cap\{t\leq t_0\}\cap\calQ_r$ at $(x_0,t_0)$, then
\begin{align*}
    -(\Delta\phi+f)(x_0,t_0)\leq 0\quad &\text{ if } u(x_0,t_0)>0\\
    (\phi_t-|\nabla\phi|^2-\vec{b}\cdot\nabla\phi)(x_0,t_0)\leq 0\quad &\text{ if } (x_0,t_0)\in\Gamma_u \hbox{ and } -(\Delta\phi+f)(x_0,t_0)>0.
\end{align*}

The function $u$ is a {\it viscosity supersolution} of \eqref{1.1} if for every $\phi\in C^{2,1}_{x,t}(\calQ_r)$ such that $u-\phi$ has a local minimum in $\{t\leq t_0\}\cap\calQ_r$ at $(x_0,t_0)$, then
\begin{align*}
    -(\Delta\phi+f)(x_0,t_0)\geq 0\quad &\text{ if } u(x_0,t_0)>0\\
    (\phi_t-|\nabla\phi|^2-\vec{b}\cdot\nabla\phi)(x_0,t_0)\geq 0\quad &\text{ if } (x_0,t_0)\in\Gamma_u,\, |\nabla\phi(x_0,t_0)|\neq 0 \text{ and }\\
    &-(\Delta\phi+f)(x_0,t_0)<0.
 \end{align*}

We say that a continuous non-negative function $u$ is a {\it viscosity solution} of \eqref{1.1} if $u$ is both a viscosity subsolution and a viscosity supersolution of \eqref{1.1}.
\end{definition}

For two vectors $\nu,\mu \in  \bbR^d \setminus\{0\}$,  the angle between them is denoted as
\[
\langle \nu,\mu\rangle:=\arccos\left(\frac{\nu\cdot\mu}{|\nu||\mu|}\right)\in [0,\pi].
\]
The half-space with interior normal $\nu$ is 
\[
H(\nu):=\{x\in\bbR^d\,|\,x\cdot\nu\geq 0\}.
\]
We denote a spacial cone to direction $\mu\in\mathbb{S}^{d-1}$ with opening $2\theta$ for $\theta\in [0,\frac\pi2]$ as
\begin{equation*}
W_{\theta,\mu} := \left\{p\in\mathbb{R}^{d}: \,\langle p,\mu\rangle \leq \theta\right\}.
\end{equation*}


Next we introduce \textit{streamlines}. They are defined as the unique solution $X(t;x_0)$ of the ODE:
\begin{equation}\label{ode}
\left\{\begin{aligned}
    &\partial_t X(t;x_0)={-}{b}(X(t;x_0)), \quad t\in \bbR,\\
    &X(0;x_0)=x_0.
    \end{aligned}\right.
\end{equation}
We write $X(t):=X(t;0)$.

\subsection{Some properties of solutions to the Poisson problem}
In this section, we recall several results from \cite{kimzhang2024} and  Dahlberg's lemma, which will be frequently used in this paper.

\begin{lemma}{\rm \cite[Lemma 2.12]{kimzhang2024}}\lb{l.2.2}
Let $f:\bbR^d\to \bbR$ be Lipschitz continuous, let $r>0$ and let $\rho:\overline{B}_{2 r}\to [0,\infty)$, $\rho\in C^2(\overline{B_{2 r}})$  be a classical solution to 
$$-\Delta \rho=f,\ \ {\rm{in}}\ B_{2 r}.$$
Then, there exists a constant $C>0$, depending only on the dimension, such that for all $x\in B_{r}$
\begin{align*}
    \rho(x)\leq C\rho(0)+Cr^2\|f\|_{L^\infty(B_{2r})},\quad\quad
    |\nabla \rho(x)|\leq Cr^{-1}\rho(0)+ Cr\| f\|_{L^\infty(B_{2r})}.
\end{align*}
\end{lemma}

\begin{definition}
Let $\theta\in [0,\frac\pi2]$, $\mu\in\bbS^{d-1}$, $\eps\in [0,1)$ and $a\geq 0$. 
We say that a continuous function $\rho:\bbR^d\supset\Omega\to \bbR$  is {\it $(\eps,a)$-monotone} 
with respect to a cone $W_{\theta,\mu}$ if for every $\eps'\geq \eps$ we have
\[
(1+a\eps)\,\rho(x)\leq \inf_{y\in B_{\eps'\sin\theta}(x)}\rho(y+\eps'\mu)
\]
\end{definition}

The following lemma says that $(\eps,a)$-monotonicity implies full monotonicity in the interior.

\begin{lemma}{\rm \cite[Lemma 3.1]{kimzhang2024}}\lb{L.2.4}
Let $f\geq 0$ be Lipschitz continuous on $\overline{B}_1$, $a\geq 0$ and $\eps,\kappa_1\in (0,1)$. There exists $C=C(d)>0$ such that the following holds for all $\eps$ small enough (depending only on $d,a,\kappa_1$). If $\rho$ is a non-negative solution to $
-\Delta \rho=f$ in $B_{\eps^{1-\kappa_1}}$, and $\rho$ is $(\eps,a)$-monotone with respect to $W_{0,\mu}$ for $\mu\in\bbS^{d-1}$,
then
\[
\nabla_\mu\rho(x)\geq a(1-C\eps^{\kappa_1})\rho(x)-C(1+a)\eps^{2-\kappa_1}\|f\|_{C^1(B_1)},\quad{\mathrm{for\ all\ }}x\in B_{\eps}. 
\]
\end{lemma}

Now, consider a Lipschitz continuous function $g: \bbR^{d-1} \to \bbR$ with $g(0)=0$ and Lipschitz constant $c_g \in(0,\frac12)$. For a fixed scale $L \geq 2$, we define a unit-width curvilinear strip $\Sigma'_L$ situated directly beneath the graph of $g$ within the ball $B_L$:
$$
\Sigma'_L := B_L \cap \{x=(x',x_d) : g(x') - 1 < x_d < g(x')\}.
$$
We further denote the lower boundary of this strip as:
$$
\partial_b \Sigma'_L := B_L \cap \{x=(x',x_d) : x_d = g(x') - 1\}.
$$

We recall the well-known Dahlberg's lemma.   
\begin{lemma}{~\rm (\cite{dah})}
Let $w_1,w_2$ be two non-negative harmonic functions in $\Sigma_L$. Assume further that $w_1=w_2=0$ along the graph of $g$. Then, there exists $C>1$ depending only on $d$ such that
\[
\frac{1}{C}\leq \frac{w_1(x',x_d)}{w_2(x',x_d)}\cdot \frac{w_2(0,3/4)}{w_1(0,3/4)}\leq C
\]
in $\left\{(x',x_d)\,|\, |x'|<1,\,|x_d|<3/4,\, x_d<g(x')\right\}$.
\end{lemma}

Within the domain of $\Sigma_L$, we introduce two non-negative functions, $w_{1,L}$ and $w_{2,L}$, which satisfy the following boundary value problems:
$$
\begin{cases}
-\Delta w_{1,L} = 0, \quad -\Delta w_{2,L} = 1 & \text{in } \Sigma'_L, \\
w_{1,L} = 1, \quad w_{2,L} = 0 & \text{on } \partial_b \Sigma'_L, \\
w_{1,L} = w_{2,L} = 0 & \text{on } \partial \Sigma'_L \setminus \partial_b \Sigma'_L.
\end{cases}
$$
The following result establishes that $w_{1,L}$ and $w_{2,L}$ are comparable, with estimates that remain uniform both up to the boundary and across all scales $L \geq 2$. This comparison is a key technical bridge; it allows us to use the established regularity theory of harmonic functions to analyze our solutions, even in the presence of the source and drift terms.

\begin{proposition}{\rm 
\cite[Proposition 4.1]{kimzhang2024}}\lb{L.2.61}
For $w_{1,L},w_{2,L}$ and $g$ given as above, let $L \geq 2$ and $c_g < \cot \theta_{d}$ for some dimensional constant $\theta_{d}$.
Then
\[
       w_{2,L}\leq C w_{1,L} \quad\text{ in }\Sigma'_{L-1} \hbox{ for some } C=C(d,c_g). 
\]
\end{proposition}

\subsection{Sup-convolution}

In what follows, we recall several properties of  sup-convolutions, first introduced by Caffarelli (see e.g. \cite{Caf87}). These estimates are essential for constructing the sub- and super-solutions in Section \ref{S3}.

For non-negative functions  $u$ in $C(B_1\times (0,T))$ and $\varphi \in C^{2,1}_{x,t}(B_1\times (0,T))$ with $0\leq \varphi \leq 1/2$,  define
\begin{equation}\label{sup}
v(x,t):=\sup_{B_{\varphi(x,t)}(x)}u(y,t)\quad\text{ in }B_{1/2}\times(0,T) .
\end{equation}

\begin{lemma}{\rm \cite[Lemma 5.2]{kimzhang2024}}\lb{L.2.7}
Suppose $-\Delta u= f\geq 0$ in $\Omega_u$ with continuous $f:\bbR^{d}\to\bbR$. Let $v$ be given by \eqref{sup}, then $v(x,t) = u(y(x,t),t)$ for some $y(x,t) \in B_{\varphi(x,t)}(x)$. Then there are dimensional constants $A_0,A_1>1$ such that if $\varphi$ satisfies 
\[
\Delta\varphi\geq \frac{A_0|\nabla\varphi|^2}{|\varphi|} \quad \hbox{ in } B_1\times (0,T),
\]
then $v$ satisfies (in the viscosity sense)
\[
-\Delta v \leq (1+A_1 \|\nabla\varphi\|_\infty)f \circ y \quad \text{ in } \{v>0\} \cap [B_{1/2}\times (0,T)].
\] 
\end{lemma}

\begin{lemma}{\rm \cite[Lemma 5.3]{kimzhang2024}}\label{L.2.8}
Let $u,v$ be as given in Lemma ~\ref{L.2.7}, where $\varphi$ satisfies 
\[
\varphi \leq \eps_1,\quad |\nabla\varphi|\leq \eps_2,\quad -1/2\leq  \varphi_t\leq \eps_3.
\]
If $u_t\leq |\nabla u|^2-{b}\cdot\nabla u $ in the viscosity sense on  $\Gamma_u\cap B_{1}\times (0,T)$, and $\eps_1,\eps_2,|\eps_3|>0$ are small enough, then 
\[
v_t\leq (1+2\eps_2)^2|\nabla v|^2{+}{b}\cdot\nabla v+\left(\eps_1\|\nabla{b}\|_\infty+2\eps_2\|{b}\|_\infty+(\eps_3+2^{-1}|\eps_3|)\right)|\nabla v|
\]
in the viscosity sense on  $\Gamma_v\cap (B_{1/2}\times (0,T))$.

\end{lemma}





\section{Interior gain of monotonicity}\lb{S3}

\subsection{Estimates on time derivative}

Since any global solution is non-decreasing along streamlines, it is generally true that $\partial_t u$ is bounded from below \cite{kimzhang2024}. The goal of this section is to obtain an upper bound for $\partial_t u$. The proof is a combination of Carleson-type estimate, an gradient estimate for harmonic functions and \cite[Proposition 4.1]{kimzhang2024}. The proposition describes that over a Lipschitz domain, one can use positive Harmonic functions to bound superharmonic functions uniformly up to the free boundary. The first two ingredients are given in \cite{CJK}. 

The solution to the classical Hele-Shaw problem is needed. The following Lemma concerns $u^h$ which solves
\begin{equation}\lb{HS}
    \left\{\begin{aligned}
        -\Delta u^h &=0 \quad &&\text{ in }\{u^h>0\}\cap\calQ_2,\\
        u^h_t&=|\nabla u^h|^2\quad &&\text{ on }\partial\{u^h>0\}\cap\calQ_2.
    \end{aligned}\right.
\end{equation}

\begin{lemma}\lb{L.7.1}
Suppose that $u^h$ is a subsolution to \eqref{HS} in $\calQ_2$, and $\Gamma_{u^h}(t) \cap B_1(0)$ can be represented by $x_n=g_t(x')$ for some Lipschitz function $g_t:\bbR^{d-1}\to\bbR$ with Lipschitz constant less than a dimensional constant $c_g$. Then there are $\delta_0,  C_0>0$ depending only on $d$ and $c_g$ such that the following holds for $0<\delta < \delta_0$:  for any $x_0\in\Gamma_{u^h}(0)\cap B_1$ and $x_1\in \Omega_{u^h}(0)\cap B_1$ such that 
\[
\tfrac{\delta}{2}\leq |x_1-x_0|,\quad d(x_1,\Gamma_{u^h}(0))\leq \delta,
\]
we have
\[
\sup_{x\in B_{\delta}(x_0)}{u^h}(x,T_0)\leq C_0 a(T_1) {u^h}(x_1,0), \quad \hbox{where } a(t):=\dfrac{{u^h}(-e_n,t)}{{u^h}(-e_n,0)}.
\]
where $T_1\in [0,T_0]$ and $T_0$ is the first time ${u^h}$ is positive in $B_{\delta}(x_0)$, namely 
\[
T_0:=\sup\{t>0\,|\, \Gamma_{u^h}(t)\cap B_{\delta}(x_0)\neq \emptyset\}.
\]
\end{lemma}
The result follows directly from the proof of Theorem 2.1 in \cite{CJK}. The conclusions stated are slightly more general:  we state the result in a local version and it is for sub- and super- solutions  to \eqref{HS} instead of solutions. The constant $a(T_0)$ comes from \cite[Lemma 2.6]{CJK}. Following from it,  we have a corollary similarly as \cite[Corollary 2.2]{CJK}. 

\begin{corollary}\lb{C.7.2}
Under the assumptions of Lemma \ref{L.7.1} and assuming in addition that $C^{-1}\leq a(\cdot)\leq C$ for some $C\geq 1$, if $x_2\in \Omega_{u^h}(0)^c$ satisfies
\[
\tfrac{\delta}{2} \leq |x_2-x_0|,\quad d(x_2,\Gamma_{u^h}(0))\leq \delta,
\]
then 
\beq\lb{7.2}
\frac{\delta^2}{{u^h}(x_1,0)}\leq C_0 \, t(x_2) ,\quad \hbox{ where }t(x):= \sup\{t\geq0\,|\, {u^h}(x,t)=0\},
\eeq
and $C_0>0$ depends on $C$.
If ${u^h} $ is a supersolution, then \eqref{7.2} hold with $\leq $ replaced by $\geq$.
\end{corollary}

The next lemma is critical, as it proves an upper bound for $\partial_t u$. The key idea is to construct certain sub- and super- solutions to \eqref{HS} so that we can adapt the proof of \cite[Lemma 8.3]{CJK}.

\begin{lemma}{\rm [Time derivative estimate]}\lb{L.7.3}
Suppose that $u$ is non-decreasing along all directions in $W_{\theta,-e_d}$ with $\theta\geq \frac{\pi}{4}$ in $\calQ_2$, the non-degeneracy property \eqref{nondeg} holds, and there exist $C>1$ and $c\in(0,1)$ such that
\beq\lb{bcond}
| \partial_t u(-e_d,t)|\leq C,\quad  u(-e_d,t)\in [C^{-1},C] \quad \forall t\in (-2c,2c).
\eeq
Then there exist $C'>0$ and $\delta\in (0,c)$ such that for all $(x_0,t_0)\in \Gamma_u\cap \calQ_{1,c}$ and $(x,t)\in \Omega_u\cap \calQ_\delta(x_0,t_0)$,
\[
\partial_{t+}u(x,t):=\limsup_{\eps\to 0^+}\tfrac{1}{\eps}(u(x,t+\eps)-u(x,t))\leq C'(1+|\nabla u(x,t)|^2).
\]
\end{lemma}

\begin{proof}
For simplicity of notations, let us assume that $(x_0,t_0)=(0,0)$. Recall \eqref{ode}. For $r\in (0,1)$ and $L:=\max\{1,2\|\nabla{b}\|_\infty\}$, let 
\[
\bar{u}(x,t):=u(x+X(t)+C_1rt e_d, t )
\]
which then satisfies 
\begin{equation}\lb{7.6}
    \left\{\begin{aligned}
        -\Delta \bar u &=\bar f (x,t)\quad  &&\text{ in }\{\bar u>0\},\\
        \bar u_t&=|\nabla \bar u|^2+\bar{b}(x,t)\cdot\nabla \bar u\quad  &&\text{ on }\partial\{\bar u>0\},
    \end{aligned}\right.
\end{equation}
where 
\beq\lb{7.13}
\bar f (x,t):=f(x+X(t)+ Lrte_d),\quad \bar{b}(x,t):={b}(x+X(t)+Lrte_d)-{b}(X(t))+Lre_d.
\eeq
We also denote 
\[
\bar{\Omega}:=\Omega_{\bar{u}}\cap \calQ_1,\quad \bar{\Omega}(t):=\Omega_{\bar{u}}(t)\cap B_1\quad\text{and}\quad \bar{\Gamma}(t)={\Gamma}_{\bar{u}}(t)\cap B_1.
\]

\noindent {\bf Step 1.} Construct sub- and super- solutions.

For each $t$, let $w_1(\cdot,t)$ be the unique non-negative harmonic function in $\bar{\Omega}(t)$ such that $w_1(\cdot,t)=0$ on $\bar{\Gamma}(t)$ and $w_1(\cdot,t)=\bar{u}(\cdot,t)$ on $\partial\bar{\Omega}\backslash \bar{\Gamma}(t)$. 
It follows from Lemma 11.12 \cite{cbook} that $\nabla_{-e_d} w_1(\cdot,t)\geq 0$ in $B_r$ for all $r,|t|$ sufficiently small.

Next, let $w_2:=\bar{u}-w_1$. It follows from Proposition \ref{L.2.61}  that there exists $C_1>1$ such that $w_2\leq (C_1-1)w_1$ and so
\beq\lb{7.3}
w_1\leq \bar{u}\leq C_1w_1.
\eeq
Moreover, this and the non-degeneracy condition yield that there exists $c>0$ such that for any $\eps>0$ sufficiently small, 
\beq\lb{7.4}
\bar{u}(x-\eps e_d,t)\geq c\eps \quad\text{ for all $x,t$ such that }x\in\bar{\Gamma}(t).
\eeq

First, we claim that $2C_1{w}_1$ is a subsolution to \eqref{HS} in $\calQ_r$ if $r>0$ is sufficiently small. Since $w_1$ is harmonic in its support, it suffices to verify the free boundary condition. Suppose $(x_1,t_1)$ is on the free boundary and $\phi\in C^{2,1}_{x,t}$ is such that $2C_1w_1-\phi$ has a local maximum at $(x_1,t_1)$ for $t\leq t_1$ and in $\{w_1>0\}$. 
Then, \eqref{7.3} imply that $2\bar{u}-\phi$ also obtains a local maximum at $(x_1,t_1)$, and therefore \eqref{7.6} yields
\beq\lb{7.5}
2\phi_t(x_1,t_1)\leq |\nabla\phi(x_1,t_1)|^2+2\bar{b}(x_1,t_1)\cdot\nabla\phi(x_1,t_1).
\eeq
Moreover, since $x_1\in\bar \Gamma(t_1)$, it follows from \eqref{7.3} and \eqref{7.4} that 
\[
\phi(x_1-\eps e_d,t_1)-\phi(x_1,t_1)\geq 2\bar u(x_1-\eps e_d,t_1)\geq 2c\eps
\]
for all $\eps>0$ sufficiently small. Thus, $|\nabla\phi(x_1,t_1)|\geq 2c$ and, by taking $r$ to be small enough,  we get from \eqref{7.13} that
\[
2|\bar b (x,t)|\leq 2(\|\nabla{b}\|_\infty(r+Lr^2) +Lr)\leq 2c\leq |\nabla \phi(x_1,t_1)|\quad\text{in }\calQ_r.
\]
Thus, \eqref{7.5} shows
\[
\phi_t(x_1,t_1)\leq |\nabla \phi(x_1,t_1)|^2
\]
which implies that $2C_1w_1$ is a subsolution to \eqref{HS}.

Second, we claim that  ${w}_1$ is a supersolution to \eqref{HS}. Indeed, let
$\phi\in C^{2,1}_{x,t}$ be such that $w_1-\phi$ has a local minimum at $(x_1,t_1)$ for $t\leq t_1$, and $\nabla\phi(x_1,t_1)\neq 0$. 
Then by \eqref{7.3}, $\bar{u}-\phi$ obtains a local minimum at $(x_1,t_1)$ and so
\beq\lb{7.7}
\phi_t(x_1,t_1)\geq |\nabla\phi(x_1,t_1)|^2+\bar{b}(x_1,t_1)\cdot\nabla\phi(x_1,t_1).
\eeq
Note that since $L\geq 2\|\nabla{b}\|_\infty$, if taking $r<L^{-1}$, for $(x,t)\in \calQ_r$,
\[
|{b}(x+X(t)+Lrt e_d)-{b}(X(t))| \leq 2r\|\nabla{b}\|_\infty\leq Lr. 
\]
This implies that 
\beq\lb{7.17}
\langle \bar{b}(x,t), e_d \rangle\leq \tfrac\pi4.
\eeq

Now, the monotonicity condition on $\bar{u}$ yields that $\phi$ is non-decreasing along directions in $W_{\theta,-e_d}$ at $(x_1,t_1)$. Therefore,  $\langle\nabla\phi(x_1,t_1), -e_d \rangle\leq  \tfrac{\pi}{2}-\theta\leq\tfrac\pi4$. This and  and \eqref{7.17} imply $
\bar{b}(x_1,t_1)\cdot\nabla\phi(x_1,t_1)\leq 0$. Hence we obtain $\phi_t(x_1,t_1)\geq |\nabla\phi(x_1,t_1)|^2$ from \eqref{7.7} which proves that $w_1$ is a supersolution.

\smallskip

\noindent {\bf Step 2.} Let us bound $(\bar u(x,\eps)-\bar u(x,0))/\eps$ for $x$ near the free boundary. Fix $x_1\in \bar\Gamma(0)\cap B_{r/2}$ and let $\delta_n>0$ satisfy $\lim_{n\to\infty}\delta_n=0$. By \eqref{7.3}, set
\[
\tau_n:=\inf\{t>0\,|\, \bar u(\delta_{n}e_d,t )>0\}=\inf\{t>0\,|\, w_1(\delta_{n}e_d,t )>0\}.
\]
Since $w_1$ is a supersolution to \eqref{HS} in $\calQ_{r}$, $\tau_n>0$ and satisfies $\tau_n\to0$ as $n\to \infty$. Actually, we can fix $\tau_n\to 0$ first, and then define $\delta_n\to 0$ so that the above equalities hold.
Using that the free boundary is a Lipschitz graph and $w_1$ is a supersolution to \eqref{HS} and $2C_1w_1$ is a subsolution, applying Corollary \ref{C.7.2} twice yields
\[
w_1(x_1-\delta_n e_d,0) \asymp \delta_{n}^2/\tau_n.
\]
Since $2C_1w_1$ is a subsolution, applying Lemma \ref{L.7.1} to $2C_1w_1$ yields
\beq\lb{7.3'}
\sup_{x\in B_{\delta_{n}}(x_1)}w_1(x,\tau_n)\leq Cw_1(x_1-\delta_n e_d,0 )\leq  C\delta_{n}^2/\tau_n.
\eeq

We define
\[
p_n(x,t):=(\bar{u}(x,t+\tau_n)-\bar{u}(x,t))/{\tau_n}.
\]
By \eqref{7.3} and \eqref{7.3'}, we get
\beq\lb{7.14}
\begin{aligned}
p_n(x_1-\delta_n e_d,0 )\leq Cw_1(x_1-\delta_n e_d ,\tau_n)/\tau_n\leq C\delta_{n}^2/\tau_n^2.
\end{aligned}
\eeq
Now, we define an interior boundary layer that is approximately $\delta_n$ away from the free boundary:
\[
N_n:=\{z\in B_{r/2}\cap \bar\Omega(0)\,|\,d(z,\bar\Gamma(0))\in [\delta_n,2\delta_n]\}.
\]
For any $z\in N_n(x_1)$,
\beq\lb{7.19}
p_n(z,0)\stackrel{\eqref{7.14}}\leq Cw_1(z,0)^2/d(z,\bar{\Gamma}(0))^2.
\eeq
Note that, due to the Lipschitz condition of the free boundary, $\nabla_{-e} w_1(\cdot,0)\geq 0$ for all $e\in W_{\theta/2,-e_d}$ in $B_r$ if $r$ is small enough by \cite[Lemma 11.12]{cbook}. Therefore, we can apply \cite[Lemma 4]{Caf87} to get
\beq\lb{7.21}
w_1(z,0)/d(z,\bar{\Gamma}(0)) \asymp |\nabla w_1(z,0)|\quad\text{ in }B_{r/2}.
\eeq

Next, let $U$ be a harmonic function such that $U=|\nabla w_1(\cdot,0)|^2$ on $\bar\Gamma(0)$ ($\nabla w_1(\cdot,0)$  is understood as the nontangential limit) and $U(-e_d)=1$. By \cite[Lemma 8.2 (e)]{CJK} and Dahlberg's lemma,
\[
|\nabla w_1(\cdot,0)|^2\leq C(U(\cdot)+w_1(\cdot,0))\leq C(U(\cdot)+\ell(\delta_n))\quad\text{ in }N_n,
\]
where $\lim_{n\to\infty}\ell(\delta_n)=0$. Due to the assumption, $|p_n(-e_d,0)|\leq C$. Then by \eqref{7.19} and Dahlberg's lemma again, we find
\beq\lb{7.16}
p_n(\cdot,0)\leq C(U(\cdot)+\ell (\delta_n))\quad\text{ in }\{x\in B_{r/2}\cap \bar\Omega(0)\,|\,d(x,\bar\Gamma(0))>\delta_n\}.
\eeq

Finally, by \cite[Lemma 8.2 (e)]{CJK} and $U(-e_d)=1$, 
\[
U(x)\leq C(1+|\nabla w_1(x,0)|^2)\quad \text{ in }\bar\Omega(0)\cap B_{r/2}.
\]
Hence, passing $\tau_n\to 0$ in \eqref{7.16}, we get
\[
\bar u_+(x,0)\leq C(1+|\nabla w_1(x,0)|^2)\leq C(1+|\nabla \bar u(x,0)|^2)\quad \text{ in }\bar\Omega(0)\cap B_{r/2}
\]
where the second inequality is due to Lemma \ref{l.2.2} (or \cite[Lemma 4.4]{kimzhang2024}), \eqref{7.3} and \eqref{7.21}. This yields the conclusion.
\end{proof}

\subsection{Interior gain of monotonicity}
From now on, we will always assume that $u$ is non-decreasing along all directions in $W_{\theta,-e_d}$ with $\theta\geq \frac{\pi}{4}$ in $\calQ_2$ and \eqref{bcond}.

Suppose that $(0,0)\in\Gamma_u$. 
We consider solutions along the streamline starting at $0$, and moreover, we zoom in near the free boundary point. For some $r,\gamma\in (0,1]$, let
\[
u_{r,\gamma}:=\frac{\gamma}{r}u(r x+ X(r\gamma t), r\gamma t)
\]
which then satisfies
\begin{equation}\lb{3.17}
    \left\{\begin{aligned}
        -\Delta u_{r,\gamma} &=f_{r,\gamma}(x,t)\quad  &&\text{ in }\{u_{r,\gamma} >0\},\\
        \partial_t u_{r,\gamma} &=|\nabla  u_{r,\gamma}|^2+{b}_{r,\gamma}(x,t)\cdot\nabla u_{r,\gamma} \quad  &&\text{ on }\partial\{u_{r,\gamma} >0\},
    \end{aligned}\right.
\end{equation}
where 
\beq\lb{3.18}
f_{r,\gamma}(x,t):=r\gamma f(r  x+ X(r\gamma t)),\quad {b}_{r,\gamma}(x,t):=\gamma\left[{b}(r x+ X(r\gamma t))-{b}(X(r\gamma t))\right].
\eeq
By the assumption, there exists $L>0$ such that for all $(x,t)\in\bbR^{d+1}$,
\beq\lb{fdel}
\|f_{r,\gamma}\|_\infty\leq Lr\gamma,\quad \|\nabla f_{r,\gamma}\|_{\infty}+\|\partial_t f_{r,\gamma}\|_\infty\leq Lr^2\gamma,
\eeq
\[
|{b}_{r,\gamma}(x,t)|\leq Lr\gamma|x|,\quad \| \nabla {b}_{r,\gamma}\|_\infty\leq Lr\gamma,\quad |\partial_t {b}_{r,\gamma}(x,t)|\leq L(r\gamma)^2|x|.
\]

Below we will show that if $u$ is monotone non-decreasing along every direction in $W_{\theta,\mu}$ and $u$ is non-degenerate near the free boundary, then in the interior of $\{u>0\}$ the solution is non-decreasing in all directions of another cone with larger opening uniformly for a short time.

\begin{lemma}{\rm [Interior gain]}\lb{L.7.4}
Under the assumptions of Lemma \ref{L.7.3}, let $c,c_1>0$ and then there exist $\theta_0,c_0,c_2>0$ such that the following holds for all $\theta,\delta$ and $\gamma$ satisfying
\[
\gamma\in (0,1],\quad 
\theta\in (\theta_0,\tfrac\pi2),\quad r,\alpha\in (0,c_0)\quad\text{and}\quad r<c_0\alpha.
\]
Let  $(x_0,t_0)\in\Gamma_u\cap \calQ_{1,c}$, and, after shifting, we assume $(x_0,t_0)=(0,0)$. Let $u_{r,\gamma}$ as before and we assume that for all $\eps\in (0,1)$,
\begin{equation}\lb{nondegn}
u_{r,\gamma}(x+\eps \mu,t)\geq c_1\gamma\eps \quad \hbox{ for all } (x,t)\in \calQ_1.
\end{equation} 

Denote $\tau_1:=c_2\alpha$, 
\[
\mu_1:=\nabla u_{r,\gamma}(\tfrac{3}{4}\mu,0)/|\nabla u_{r,\gamma}(\tfrac{3}{4}\mu,0)|,
\]
and, for any $\nu\in W_{\theta/2,\mu}$ satisfying $|\nu|<c_0$, set
\[
\delta_\nu:=|\nu|\sin(\theta/2)\quad\text{and}\quad
\eta_\nu :=c_2\cos(\langle \mu_1,\nu\rangle+\tfrac\theta2).
\]
Then, if 
\[
\alpha\leq \tfrac{\pi}{2}-\tfrac{\theta}{2}-\langle \mu_1,\nu\rangle,\quad u_{r,\gamma}(\tfrac{3}{4}\mu,0)\leq 1,
\]
we have
\[
0<\sup_{y\in B_{(1+c_2\eta_\nu ) \delta_\nu}(x)}u_{r,\gamma}(y- \nu,t)\leq  (1-\eta_\nu  \delta_\nu)u_{r,\gamma}(x,t)
\]
for all $(x,t)\in B_{1/8}(\tfrac{3}{4}\mu)\times (-\tau_1,\tau_1)$.

\end{lemma}

We emphasize that the uniform interior gain is for a short time interval $(t_0-\tau_1,t_0+\tau_1)$ depending on $\alpha$ and $|\nu|$, and this is critical. The size of the time interval is different from those in \cite{kim2,CJK}. We shall select the parameter $\gamma$ very carefully to obtain the regularity of the free boundary.

\begin{proof}

We divide the proof into 3 steps. 

\smallskip

\noindent{\bf Step 1.} Interior gain at $t=0$.  Let $\nu\in W_{\theta/2,\mu}$ with $|\nu|\in (0,\frac18)$ and let $x_1:=\tfrac{3}{4}\mu$. 
Since $\theta\geq \theta_0$, we have
\[
\delta_\nu:=|\nu|\sin(\theta/2) \asymp |\nu|.
\]
Next, take any $x\in B_{1/4}(x_1)$ and $y\in B_{\delta_\nu}(x)$, and define $\bar\nu:=\nu-(y-x)$ (then $|\bar\nu-\nu|\leq \delta_\nu$). Then $\langle\bar\nu,\nu\rangle\leq\tfrac\theta2$.

Since $u_{r,\gamma}$ is monotone along all directions in $W_{\theta,\mu}$, we also get $\nabla_{\bar\nu}u_{r,\gamma}\geq 0$. Since $\mu_1:=\nabla u_{r,\gamma}(x_1,0)/|\nabla u_{r,\gamma}(x_1,0)|$, we have 
$\langle \mu_1,\mu\rangle\leq \frac\pi2-\theta$.

By applying Lemma \ref{l.2.2} to $\nabla_{\bar \nu} u_{r,\gamma}$ and $|\nu|\asymp\delta_\nu$, we obtain
\[
\begin{aligned}
\nabla_{\bar\nu} u_{r,\gamma}(x,0)\geq c\delta_\nu\cos\langle \mu_1,\bar\nu\rangle|\nabla u_{r,\gamma}(x_1,0)|-C\delta_\nu r^2\gamma.
\end{aligned}
\]
After assuming $r$ to be small enough, we can apply \cite[Lemma 4.4]{kimzhang2024} to get
\beq\lb{7.8}
\begin{aligned}
\nabla_{\bar\nu} u_{r,\gamma}(x,0)    &\geq c\delta_\nu\cos\langle \mu_1,\bar\nu\rangle  u_{r,\gamma}(x_1,0)-C\delta_\nu r^2\gamma\\
&\geq \delta_\nu (c\alpha \, u_{r,\gamma}(x_1,0)-C r^2\gamma),
\end{aligned}
\eeq
where the second inequality is due to
\[
\langle\mu_1,\bar\nu\rangle\leq\langle\mu_1,\nu\rangle+\langle\nu,\bar\nu\rangle\leq \tfrac{\pi}{2}-\alpha.
\]

Note that $B_{1/2}(x_1)\subseteq \Omega_{u_{r,\gamma}}(t)$ for all $t\in (-c,c)$, the non-degeneracy assumption yields for some $c_1>0$,
\beq\lb{7.10}
u_{r,\gamma}(x,t)\geq c_1\gamma \quad\text{ for all }(x,t)\in  B_{1/2}(x_1)\times (-c,c) .
\eeq
Thus, it follows from \eqref{7.8} that if $r^2\ll \alpha$,
\[
\inf_{B_{1/4}(x_1)}\nabla_{\bar\nu} u_{r,\gamma}(\cdot,0)  \geq c\delta_\nu \cos\langle \mu_1,\bar\nu\rangle  u_{r,\gamma}(x_1,0)\quad\text{ for some $c>0$ independent of $\delta_\nu ,r,\alpha$,}
\]
and it follows from  Lemma \ref{l.2.2} and \eqref{fdel} that  for $x\in B_{1/4}(x_1)$,
\[
u_{r,\gamma}(x_1,0)\leq Cu_{r,\gamma}(x,0).
\]
These yield for $x\in B_{1/4}(x_1)$ and $y\in B_{\delta_\nu}(x)$,
\beq\lb{7.18}
\begin{aligned}
u_{r,\gamma}(y-\nu,0)&=u_{r,\gamma}(x-\bar\nu,0)\leq u_{r,\gamma}(x,0)-c{\delta_\nu} \cos\langle \mu_1,\bar\nu\rangle u_{r,\gamma}(x_1,0)\\
&\leq u_{r,\gamma}(x,0)-2\eta_\nu {\delta_\nu} u_{r,\gamma}(x,0),
\end{aligned}
\eeq
where, in the last inequality, we used that
\[
\langle\mu_1,\bar\nu\rangle\leq\langle\mu_1,\nu\rangle+\langle\nu,\bar\nu\rangle\leq \langle \mu_1,\nu\rangle+\tfrac\theta2,
\]
and defined
\[
\eta_\nu :=\tfrac{c}2\cos(\langle \mu_1,\nu\rangle+\tfrac\theta2)\leq \tfrac{c}2\cos\langle\mu_1,\bar\nu\rangle.
\]
Since $\langle \mu_1,\nu\rangle\leq \tfrac{\pi}{2}-\tfrac\theta2-\alpha$,  we also have $\eta_\nu \geq c'\alpha$.

By Lemma \ref{l.2.2} again, we get 
\beq\lb{7.11}
\sup_{B_{1/2}(x_1)}|\nabla u_{r,\gamma}(\cdot,0)|\leq C\inf_{B_{1/2}(x_1)}u_{r,\gamma}(\cdot,0)+Cr\gamma.
\eeq
Then,  for $x\in B_{1/4}(x_1)$ and $z:=y-x\in B_{\delta_\nu}(0)$, and for any $c_2\in(0,1)$,
\begin{align*}
&u_{r,\gamma}(x+(1+c_2\eta_\nu )z-\nu,0)-u_{r,\gamma}(x,0)\\
&\qquad\qquad \stackrel{\eqref{7.18}}\leq u_{r,\gamma}(x+(1+c_2\eta_\nu )z-\nu,0)-u_{r,\gamma}(x+z-\nu,0)-2\eta_\nu {\delta_\nu} u_{r,\gamma}(x,0)\\
&\qquad\qquad \leq c_2\eta_\nu |z|\sup_{B_{1/2}(x_1)}|\nabla u_{r,\gamma}(\cdot,0)|-2\eta_\nu {\delta_\nu} u_{r,\gamma}(x,0)\\
&\qquad\qquad \leq Cc_2\eta_\nu {\delta_\nu} u_{r,\gamma}(x,0)+Cc_2 \eta_\nu  \delta_\nu r\gamma-2\eta_\nu {\delta_\nu} u_{r,\gamma}(x,0).
\end{align*}
Since $u_{r,\gamma}(x,0)\geq c_1\gamma$ by \eqref{7.10},
if $r\ll c_1$ and $c_2\ll 1/C$,
\begin{align*}
u_{r,\gamma}(x+(1+c_2\eta_\nu )z-\nu,0)-u_{r,\gamma}(x,0) \leq -\eta_\nu {\delta_\nu} u_{r,\gamma}(x,0)
\end{align*}
In all, this gives for all $\nu\in W_{\theta/2,\mu}$ and $x\in B_{1/4}(x_1)$,
\beq\lb{7.12}
\sup_{B_{(1+c_2\eta_\nu ){\delta_\nu}}(x)}u_{r,\gamma}(\cdot-\nu,0)\leq (1-\eta_\nu {\delta_\nu})u_{r,\gamma}(x,0).
\eeq


\smallskip

\noindent{\bf Step 2.} Interior gain for a time interval. Note that
\[
\partial_{t+} u_{r,\gamma}(x,t)=\gamma^2(\partial_{t+} u+(\nabla u)\cdot b(X))(r x+ X(r\gamma t), r\gamma t).
\]
Hence, by Lemma \ref{L.7.3}, we have
\[
\partial_{t+} u_{r,\gamma}\leq C(\gamma^2+|\nabla u_{r,\gamma}|^2)\quad\text{ in } B_{1/2}(x_1).
\]
Using the assumption on the growth rate of $u$, the same inequality holds with $u_{r,\gamma}$ replaced by $-u_{r,\gamma}$. 
By \eqref{7.10} and \eqref{7.11}, we have
\[
\partial_{t+} u_{r,\gamma},\quad \partial_{t+} (-u_{r,\gamma})(x,t)\leq  C \inf_{B_{1/2}(x_1)}u_{r,\gamma}(\cdot,t)^2.
\]
This and Lemma \ref{l.2.2} imply that for $\tau_1:=c_3\delta_\nu\alpha$ for some $c_3\in (0,1)$, 
\[
|u_{r,\gamma}(x,t)-u_{r,\gamma}(x,0)|\leq Cc_3\delta_\nu\alpha\, u_{r,\gamma}(x_1,0) u_{r,\gamma}(x,t)\leq Cc_3 \delta_\nu \alpha \, u_{r,\gamma}(x,t),
\]
for all $t\in [-\tau_1,\tau_1]$ and $x\in B_{1/2}(x_1)$. Here, in the second inequality, we used the assumption that $u_{r,\gamma}(\tfrac{3}{4}\mu,0)\leq 1$.
Therefore, by \eqref{7.12} and $\eta_\nu \geq c\alpha$,
for $|t|\leq \tau_1$ and $x\in B_{1/4}(x_1)$, for some $C>1$ we have
\begin{align*}
0<\sup_{B_{(1+c_2\eta_\nu )\delta_\nu}(x)}u_{r,\gamma}(\cdot- \nu,t)&\leq (1-\eta_\nu \delta_\nu+
Cc_3M \delta_\nu \gamma\alpha )u_{r,\gamma}(x,t)\\
&\leq (1-\eta_\nu \delta_\nu/2)u_{r,\gamma}(x,t),    
\end{align*}
where the second inequality holds as $\eta_\nu \geq C\alpha$. 

\smallskip

\noindent {\bf Step 3.} Improved interior monotonicity. Let us fix $\eps_0\in (0,1)$. It follows from Step 2 that for any $\nu\in W_{\theta/2,\mu}$ such that $\delta_\nu=\eps_0$, we have 
for $|t|\leq \tau_1=c_3\eps_0 \alpha$ and $x\in B_{1/4}(x_1)$, 
\[
\sup_{y\in B_{(1+c_2\eta_\nu ) \eps_{0}}(x)}u_{r,\gamma}(y- \nu,t)\leq  (1-\eta_\nu  \eps_0/2)u_{r,\gamma}(x,t)
\]
where
\[
\eta_\nu :=\tfrac{c}2\cos(\langle \mu_1,\nu\rangle+\tfrac\theta2).
\]
This shows that $u_{r,\gamma}$ is  $(\eps,c\eta_\nu )$-monotone increasing along the direction 
\[
\xi\in B_{(1+c_2\eta_\nu )\eps_0}(\nu)\quad\text{  for all $\eps\geq C\eps_0$. }
\]
Thus, if $\eps_0$ is sufficiently small, by Lemma \ref{L.2.4} with $\kappa_1=1/2$, we obtain that
for all $(x,t)\in B_{1/8}(x_1)\times [-\tau_1,\tau_1]$,
\[
\nabla_\xi u_{r,\gamma}(x,t)/|\xi|\geq c\eta_\nu  |\xi| u_{r,\gamma}(x,t)-C\eps_0^{3/2}r\gamma.
\]
Since $|\xi|\asymp \eps_0$, \eqref{7.10} and $\eta_\nu \geq c\alpha\gg r$, it follows that
\[
\nabla_\xi u_{r,\gamma}\geq c\eta_\nu  \eps_0 |\xi| u_{r,\gamma}.
\]
It follows from that for all $\eps$ sufficiently small and $|\nu|\sin(\theta/2)=\eps$,
\[
\sup_{y\in B_{(1+c_2\eta_\nu ) \eps}(x)}u_{r,\gamma}(y- \nu,t)\leq  (1-\eta_\nu  \eps)u_{r,\gamma}(x,t)
\]
for all $(x,t)\in B_{1/8}(x_1)\times [-c_3\eps_0 \alpha,c_3\eps_0 \alpha]$ (possibly with the constant $c$ in $\eta_\nu $ replaced by a smaller one depending on $\eps_0$).
\end{proof}

\section{Continuity of the free boundary}\lb{S4}

\subsection{Propagate the gain to the boundary}
In this section, we propagate the interior regularity gain of regularity to the free boundary and show an improved regularity of the free boundary.
We start with introducing a family of text functions. The result can be found in \cite[Lemma 3.2]{kim2}.

\begin{lemma}{\rm [A family radius functions]}\lb{L.7.5}
For a given constant $A>0$ there are $k,C>0$ such that for all sufficiently small $\eta\geq 0$ and any $\tau>0$ there exists a $C^2$ function $\varphi_\eta(x,t)$ defined in
\[
D:= (B_1(0)-B_{1/8}(3\mu/4))\times (-\tau,\tau)
\]
such that
\begin{enumerate}
    \item $\varphi_\eta\in [1,1+\tau\eta ]$ in $D$,
    \smallskip
    \item $\varphi_\eta\Delta\varphi_\eta\geq A|\nabla\varphi_\eta|^2$ in $D$,
    \smallskip
    \item $\varphi_\eta\equiv 1$ outside $B_{8/9}(0)\times (-7\tau/8,\tau)$,
    \smallskip
    \item $\varphi_\eta\geq 1+\tau k\eta $ in $B_{1/2}(0)\times (-\tau/2,\tau)$,
    \smallskip
        \item $|\nabla\varphi_\eta|\leq C\eta $ and $0\leq \partial_t \varphi_\eta\leq C\eta $ in $D$.
\end{enumerate}
\end{lemma}

In the following lemma, we prove interior enlargement of the monotonicity cone. In it,  we take constant $A=A_0$ from \cite[Lemma 5.2]{kimzhang2024} and then $\eta, k$ from Lemma \ref{L.7.5}.

\begin{lemma}{\rm [A comparison result]}\lb{L.7.6}
Let $u_1$ and $u_2$ be viscosity solutions to \eqref{1.1} in $ B_1\times [-\tau,\tau]$ with $f_1,{b}_1$ and $f_2,{b}_2$ in place of $f,{b}$, respectively. Let $r,\gamma\in(0,1]$.
We assume that 
\begin{enumerate}
    \item 
$f_i,b_i$ with $i=1,2$ satisfy \eqref{fdel}, and
\[
|f_1-f_2|\leq Lr^2\gamma\eps\, \quad\text{ and }\quad |{b}_1-{b}_2|\leq Lr\gamma\eps.
\]
\item $\Gamma_{u_1}$ are Lipschitz in space with small Lipschitz constant.
\item For some $c_1>0$, $u_{1}(x+h\mu,t)\geq c_1\gamma h$ on $\Gamma_{u_1}$ for all $h\in (0,1)$.
\item $(0,0)\in \Gamma_{u_2}$. 
\item The following monotonicity condition holds:
\beq\lb{72.1}
v_\eps(x,t):=\sup_{y\in B_\eps(x)}u_1(y,t)\leq u_2(x,t)\quad\text{ in }B_1\times [-\tau,\tau] .
\eeq
\end{enumerate}
If  $C\eta\geq r\gamma$ for some $C>0$, there exists $C_1>0$ sufficiently large such that if
\beq\lb{72.2}
v_{(1+\tau\eta )\eps}(x,t)\leq (1-C_1\eps\eta/\gamma) u_2(x,t)\quad\text{ in } B_{1/8}(\tfrac{3}{4}\mu) \times [-\tau,\tau],
\eeq
then
\[
v_{(1+\tau k\eta )\eps}(x,t)\leq u_2(x,t)\quad\text{ in } B_{1/2}\times [-\tau/2,\tau].
\]
\end{lemma}

\begin{proof}
Let us define a perturbation function of $u_1$ using sup-convolusions. Let $\varphi_\eta$ and $k>0$ from Lemma \ref{L.7.5} and denote $x_1:=\frac{3}{4}\mu$. Define 
\[
v(x,t):=v_{\eps\varphi_\eta(x,t)}(x,t).
\]
And each $t\in [-\tau,\tau]$, we set $w_1(x,t)$ to be the unique solution to
\[
\left\{
\begin{aligned}
    -\Delta w_1(x,t)&=0 \quad &&\text{ in }\Omega_v\cap (\partial B_1\times [-\tau,\tau] ),\\
    w_1&=0\quad &&\text{ on }\{v=0\}\cup (\partial B_1\times [-\tau,\tau] ),\\
    w_1&=u_2\quad &&\text{ on }\partial B_{1/8}(x_1) \times [-\tau,\tau],
\end{aligned}
\right.
\]
and $w_2(x,t)$ to be the unique solution to
\[
\left\{
\begin{aligned}
    -\Delta w_2(x,t)&=Lr\gamma \quad &&\text{ in }\Omega_v\cap (\partial B_1\times [-\tau,\tau] ),\\
    w_2&=0\quad &&\text{ on }\{v=0\}\cup (\partial B_1\times [-\tau,\tau] ),\\
    w_2&=0\quad &&\text{ on }\partial B_{1/8}(x_1) \times [-\tau,\tau].
\end{aligned}
\right.
\]
Then for some $C_1$ sufficiently large to be determined, define
\beq\lb{eps2}
\eps_2:=C_1\eps(\eta+r),\quad \eps_1:=C_1\eps_2+C_1\eta\eps/\gamma\quad\text{and}
\eeq
\[
\bar v(x,t):=v(x,t)+\varepsilon_1 w_1(x,t)-\varepsilon_2 w_2(x,t).
\]
Here the role of $w_1$ is to propagate the gain from the interior, while $w_2$ is needed to compensate the error coming from the perturbation of the source term $f$.

Note that $\varphi_\eta=1$ on $(( B_1\backslash B_{8/9})\times [-\tau,\tau] )\cup ( B_1\times \{-r\} )$. So, \eqref{72.1} yields $v\leq u_2$ in this region,  and so for each $t\in [-\tau,\tau]$,
\beq\lb{order1}
\bar v(\cdot,t)\leq v(\cdot,t)\leq u_2(\cdot,t)\quad\text{ on } \partial B_1\,\text{ and }\, (B_1\backslash B_{8/9})\cap \Gamma_{\bar v}(t)
\eeq
Next, since $\varphi_\eta\leq 1+r\eta$ and \eqref{eps2}, \eqref{72.2} implies that
\[
\bar v\leq v_{(1+r\eta )\eps}+\varepsilon_1 u_2\leq u_2\quad\text{ on } B_{1/8}(x_1) \times [-\tau,\tau].
\]
Since $v\leq u_2$ on $ B_1\times \{-r\} $, $\bar v(-r,\cdot)$ is supported in the positive set of $u_2$. Since $\bar v\leq u_2$ on $(\partial B_1\cup \partial B_{1/8}(x_1))\times \{-r\}$, if we are able to show that $-\Delta\bar v\leq -\Delta u_2=f_2$ on $(B_1\backslash B_{1/8}(x_1))\times\{-\tau\}$, then, by the comparison principle, we can conclude that
\[
\bar v\leq u_2\quad\text{ on } (B_{1}\backslash B_{1/8}(x_1) )\times \{-r\}.
\]

Now, we show $-\Delta\bar v\leq f_2$ in $(B_1\backslash B_{1/8}(x_1))\times [-\tau,\tau]$. 
Indeed, suppose that $v(x,t)=u_1(y(x,t),t)$ and so $|y(x,t)-x|\leq \eps\|\varphi_\eta\|_\infty\leq 2\eps$. It follows from \cite[Lemma 5.2]{kimzhang2024} that
\begin{align*}
-\Delta \bar v(x,t)&=-\Delta  v(x,t)+\varepsilon_2\Delta w_2(x,t)\\
&\leq (1+A_1\eps\|\nabla\varphi_\eta(\cdot,t)\|_\infty)(f_1(y(x,t),t))-\varepsilon_2Lr\gamma\\
&\leq (1+CA_1\eps\eta )(f_1(x,t)+2Lr^2\gamma\eps )-\varepsilon_2Lr\gamma\\
&\leq (1+CA_1\eps\eta )(f_2(x,t)+3Lr^2\gamma\eps)-C_1\eps(\eta+r)Lr\gamma
\end{align*}
which is no more than $f_2(x,t)$ due to \eqref{eps2} and $f_2\leq Lr\gamma$ by \eqref{fdel}. Here $C_1$ is taken to be large enough depending only on $CA_1$.

In view of \eqref{order1}, we then show $\bar v\leq u_2$ in $B_{8/9}\backslash B_{1/8}(x_1)\times [-\tau,\tau]$. Due to the comparison principle, it remains to show that  $\bar v$ is a viscosity subsolution in the region to \eqref{1.1} with $f_2,{b}_2$ in place of $f,{b}$. For this purpose,  it remains to show that $\bar v$  satisfies the desired inequality at the free boundary points in $B_{8/9}\times [-\tau,\tau]$. Indeed, let $\phi(x,t)\in C^{2,1}_{x,t}$ be a test function such that $\bar{v}-\phi$ has a local maximum zero at $(z_0,t_0)\in \Gamma_{\bar{v}}$ with $|z_0|< 8/9$ for $t\leq t_0$.

Since $w_1=u_2\geq\frac{c_1}2 \gamma$ on $\partial_{1/8}(x_1)$ and $w_1(\cdot,t)$ is harmonic in $(B_1\backslash B_{1/8}(x_1))\cap \Omega_v(t)$ and the free boundary is a Lipschitz graph with small Lipschitz constant, 
it follows from Proposition \ref{L.2.61} that there exists $c>0$ such that 
\beq\lb{7.31}
w_1\geq c(w_2+v)\quad\text{ in }B_{8/9}\backslash B_{1/8}(x_1)\times [-\tau,\tau].
\eeq
Also using that $\eps_2\ll \eps_1$, there exists $c_2>0$ such that
\[
\bar v=v+\eps_1w_1-\eps_2w_2\geq (1+c_2\varepsilon_1)v\quad\text{ in }B_{8/9}\backslash B_{1/8}(x_1)\times [-\tau,\tau].
\]
Consequently, $(1+c_2\varepsilon_1)v-\phi$ obtains a local maximum zero at the free boundary point $(z_0,t_0)$. Hence, Lemma \ref{L.2.8} and Lemma \ref{L.7.5} yield, at $(z_0,t_0)$,
\beq\lb{7.32}
\begin{aligned}
\phi_t&\leq (1+c_2\varepsilon_1)^{-1}(1+C\eps\eta)^2|\nabla \phi|^2+{b}_1\cdot \nabla \phi+C\eps(\eta+r\gamma)|\nabla \phi|\\
&\leq |\nabla \phi|^2+{b}_1\cdot \nabla \phi-\frac{c_2\varepsilon_1}{2}|\nabla\phi|^2+C\eps(\eta+r\gamma)|\nabla \phi|\\
&\leq |\nabla \phi|^2+{b}_2\cdot \nabla \phi-\frac{c_2\varepsilon_1}{2}|\nabla\phi|^2+c_1c_2\gamma\varepsilon_1|\nabla \phi|,   
\end{aligned}
\eeq
where, in the last inequality, $c_1$ is to be determined and we apply $C\eta\geq r\gamma$ and \eqref{eps2}. 
More precisely, in \eqref{7.32}, for some universal constant $C'$ we can replace $C\eps(\eta+r\gamma)$ by
\beq\lb{e.b}
\|b_1-b_2\|_\infty+C'\eta\eps\|b_1\|_\infty+2\eps\|\nabla b_1\|_\infty.
\eeq

Since $u_1$ is uniformly non-degenerate near the free boundary, $v$ is non-degenerate near the free boundary. Consequently, for all $h>0$ small, there exists $c_1>0$ such that
\[
\bar v(z_0+h\mu,t_0)\geq c_1\gamma h.
\]
This implies that $|\nabla\phi(z_0,t_0)|\geq c_1\gamma$. Thus, it follows from \eqref{7.32} that
$\phi_t\leq |\nabla \phi|^2+{b}_2\cdot \nabla \phi$ at $(z_0,t_0)$.

Overall, we proved that $\bar v$ is a subsolution and, by the comparison principle,  we can conclude that $\bar v\leq u_2$ in $(B_1\backslash B_{1/8}(x_1))\times [-\tau,\tau]$. 
Hence, in $B_{1/2}\times (-\tau/2,\tau)$,
\[
v_{\eps\varphi_\eta}\stackrel{\eqref{7.31}}\leq v-\varepsilon_2w_2+\varepsilon_1w_1 =\bar v \leq u_2.
\]
The  assertion of the lemma follows immediately as $\varphi_\eta\geq 1+\tau k\eta$.
\end{proof}

Lemma \ref{L.7.6} is used to obtain the following result, which will be applied iteratively to obtain the spatial regularity of the free boundary. It shows that, after knowing that the monotonicity of the solution is improved uniformly in the interior over a small time interval, then the monotonicity improves uniformly up to the free boundary over a smaller region. 

\begin{lemma}{\rm [Gain up to the boundary]}\lb{L.8.3}
Under the assumptions of Lemma \ref{L.7.4},  
suppose $(0,0)\in\Gamma_u\cap \calQ_{1,c}$. Let $c_0,\alpha,\tau_1$, $\mu_1$, $r,\gamma$ and $\delta_\nu,\eta_\nu$  for any $\nu\in W_{\theta/2,\mu}$ from the lemma. 
Then there exist $C,c_3>0$ such that if 
\[
\eta_\nu\geq Cr,\quad \alpha\leq \tfrac{\pi}{2}-\tfrac{\theta}{2}-\langle \mu_1,\nu\rangle,\quad u_{r,\gamma}(\tfrac{3}{4}\mu,0)\leq 1,
\]
we have 
\[
\sup_{y\in B_{(1+c_3\alpha\eta_\nu\gamma)\eps}(x,t)}u_{r,\gamma}(y-\nu,t) \leq u_{r,\gamma}(x,t)\quad\text{ in } B_{1/2}\times [-\tau_1/2,\tau_1]
\]
where $\eps:=|\nu|\sin(\tfrac{\theta}{2})$.
\end{lemma}

\begin{proof}
Let $r,\gamma, \alpha$ satisfy the conditions of Lemma \ref{L.7.4}. 
It follows from Lemma \ref{L.7.4} that 
there is $c_2\in (0,1)$ in dependent of $r,\gamma$ and $\alpha$ such that
 for any $\nu\in W_{\theta/2,\mu}$ satisfying $|\nu|\ll1$ and
 \[
\alpha\leq \tfrac{\pi}{2}-\tfrac{\theta}{2}-\langle \mu_1,\nu\rangle, 
\] 
 we have 
\beq\lb{7.41}
\sup_{y\in B_{(1+c_2\eta_\nu ) \eps}(x)}u_{r,\gamma}(y- \nu,t)\leq (1-\eta_\nu\eps) u_{r,\gamma}(x,t)\quad\text{ in } B_{1/8}(x_1)\times (-\tau_1,\tau_1)
\eeq
where
\[
\eta_\nu:=c_2\cos(\langle \mu_1,\nu\rangle+\theta/2),\quad\eps:=|\nu|\sin(\tfrac{\theta}{2})\quad\text{and}\quad \tau_1:=c_2\alpha
\]
and $x_1:=\frac{3}{4}\mu$, $\mu_1:=\nabla u_{r,\gamma}(x_1,0)/|\nabla u_{r,\gamma}(x_1,0)|$.

Now, we let $\varphi_\eta$ from Lemma \ref{L.7.5} with $\tau:=\tau_1$ and $\eta:=c\eta_\nu \gamma$, where $c:=\min\{c_2,1/C_1\}$ with $C_1$ from Lemma \ref{L.7.6}. Define
\[
u_1(x,t):=u_{r,\gamma}(x-\nu,t)\quad\text{and}\quad u_2(x,t):=u_{r,\gamma}(x,t).
\]
Then $u_1$ and $u_2$ are viscosity solutions to \eqref{1.1} in $ B_1\times [-\tau,\tau]$ with $f_1,{b}_1$ and $f_2,{b}_2$ in place of $f,{b}$, respectively, where
\[
f_1(x,t)=f_{r,\gamma}(x-\nu,t),\quad b_1(x,t)=b_{r,\gamma}(x-\nu,t),\quad
f_2=f_{r,\gamma},\quad b_2=b_{r,\gamma}.
\]
Since $|\nu|\asymp \eps$, the conditions in Lemma \ref{L.7.6} on the coefficients are satisfied. 
Since $u_{r,\gamma}(\cdot,t)$ is non-decreasing along all directions of $W_{\theta,\mu}$, the definition of $\eps$ yields
\[
\sup_{y\in B_{\eps }}u_1(y,t)\leq u_2(x,t).
\]
As a consequence of \eqref{7.41} and $\tau_1(c\eta_\nu\gamma)\leq c_2\eta_\nu$, for $(x,t)\in B_{1/8}(x_1)\times (-\tau_1,\tau_1)$,
\begin{align*}
\sup_{y\in B_{(1+\tau_1\eta)\eps}(x,t)}u_1(y,t)&
\leq \sup_{y\in B_{(1+c_2\eta_\nu)\eps}(x,t)}u_{r,\gamma}(y-\nu,t)\\
&\leq  (1-\eta_\nu\eps)u_{r,\gamma}(y-\nu,t)\leq (1-C_1\eta/\gamma)u_2(x,t).    
\end{align*}
Hence, the conditions of Lemma \ref{L.7.6} all hold and it implies that, if $C_1\eta\geq r\gamma$,
\[
\sup_{y\in B_{(1+\tau k\eta)\eps}(x,t)}u_1(y,t) \leq u_2(x,t)\quad\text{ in } B_{1/2}\times [-\tau/2,\tau],
\]
yielding, for $c':=cc_2$,
\[
\sup_{y\in B_{(1+c'\alpha\eta_\nu\gamma)\eps}(x,t)}u_{r,\gamma}(y-\nu,t) \leq u_{r,\gamma}(x,t)\quad\text{ in } B_{1/2}\times [-\tau/2,\tau].
\]
\end{proof}

\subsection{Iteration}

In this section, we prove that the free boundary is $C^1$ in space. We will need the following lemma for superharmonic functions  in Lipschitz domains.

It basically says that, assuming $(0,0)\in\Gamma_u$, then $\frac{u(x,0)}{d(x,\Gamma_u(0))}$ can not grow too fast as $x$ approaches the free boundary point $0$ along non-tangential directions.

\begin{lemma}\lb{Lipdomain}
Let $g:\bbR^{d-1}\to\bbR$ be a Lipschitz continuous function with Lipschitz constant $c_g>0$ small enough such that $g(0)=0$. Let $\rho\geq 0$ solve
\[
-\Delta \rho=f\text{ in }B_2\cap \{x_d<g(x')\},\quad \rho=0\text{ on }B_2\cap \{x_d=g(x')\}.
\] 
For $i=0,1$, $K\geq 2$ and $\kappa\in(0,1)$, if defining
\[
\gamma_{i,\kappa}:=
\begin{cases}
\left[K^i\rho(-K^{-i}\kappa e_d)\right]^{-1} &\text{ if }\rho(-K^{-i}e_d)\leq K^{-i},\\
1 &\text{ otherwise,}
\end{cases}
\]
then if $K$ is sufficiently large but independent of $\kappa$, we have
\[
{K\gamma_{1,\kappa}}\geq{2\gamma_{0,\kappa}}.
\]
\end{lemma}
\begin{proof}
Since $\gamma_{i,\kappa}\leq 1$, without loss of generality, we can assume that $\gamma_{1,\kappa}<1$. Define
\[
w(x):= \gamma_{0,\kappa}\rho(\kappa x)
\]
which then satisfies
\[
-\Delta w (x)= \frac{\kappa^2}{\gamma_{0,\kappa}}f(\kappa x)
\quad\text{and}\quad w(-e_d)=1.
\]

We also let $w_1$ be harmonic in $\Omega_w$ such that $w_1=w$ over $\partial (\Omega_w\cap B_1)$. Then, it follows from \cite[Corollary 4.2]{kimzhang2024} that $w\leq Cw_1$ in $B_{1/2}$. Then by \cite[Lemma 2.10]{kimzhang2024}, there exists $\beta\in (1,2)$ depending only on $d$ and $c_g$ such that
\[
w(-K^{-1}e_d)\leq Cw_1(-K^{-1}e_d)\leq  CK^{-2+\beta}.
\]
Thus, we get
\[
\gamma_{1,\kappa}
=\frac{\gamma_{0,\kappa}}{K w(-K^{-1}e_d)}\geq \gamma_{0,\kappa}K^{1-\beta}/C,
\]
which yields the conclusion for $K$ large enough depending only on $d$ and $c_g$.
\end{proof}

The proof of Theorem \ref{T.8.4} requires an intermediate cone theorem stronger than the one in \cite[Theorem 4.2]{cbook}. We specifically need to address the lack of gain from sides nearly perpendicular to $\tilde\mu$. Since this specific gain is negligible compared to that of the others, the latter remains sufficient to ensure a larger cone. Additionally, we track the explicit dependence on a small parameter $\kappa$ associated with the gain. 

While there are technical differences and improvement, the essential idea of the argument is similar to the one of \cite[Theorem 4.2]{cbook}. 

\begin{lemma}\lb{L.cone}
{\rm [Intermediate cone]} There exist $c_4,c_5>0$ such that the following holds for all $\tfrac{\pi}{4}  < \theta < \tfrac{\pi}{2}$ and $\kappa\in (0,1)$. For unit vectors $\widetilde\mu,e$, let $H(\widetilde\mu):=\{x\,|\,x\cdot\wu\geq 0\}$ and assume that the cone $W_{\theta, {e}} \subset H(\widetilde\mu)$. For any $\sigma \in W_{\theta/2, {e}}$, 
set
\[ 
\eta_\sigma :=  \cos(\langle\sigma, \widetilde\mu\rangle+\frac\theta2), \quad  \rho_\sigma:=|\sigma|\sin(\frac{\theta}{2})(1+\kappa\eta_\sigma),
\]
and for some $\alpha\geq 0$,
\[
E_\sigma:=\frac{\pi}{2}-\frac{\theta}2-\langle \sigma,\widetilde{\mu}\rangle, \quad S_{\kappa,\alpha} := \Big[\bigcup_{\sigma \in W_{\theta/2, {e}} \& E_\sigma\geq\alpha} B_{\rho_\sigma}(\sigma)\Big]\bigcup W_{\theta,{e}}.
\]
If $\alpha\leq c_4(\frac{\pi}{2}-\theta)$,
then there exist $\bar{\theta}\in (0,\frac\pi2)$ and $\bar\mu\in\bbR^d$ such that
\[
W_{\bar\theta , \bar\mu} \subset S_{\kappa,\alpha} 
\quad
\text{and}
\quad 
\frac{\pi}{2} - \bar\theta  \leq (1-c_5\kappa) \left( \frac{\pi}{2} - \theta \right). 
\]
\end{lemma}

\begin{proof}

\begin{figure}[h]
\label{fig1}
    \centering
    \includegraphics[scale=0.3]{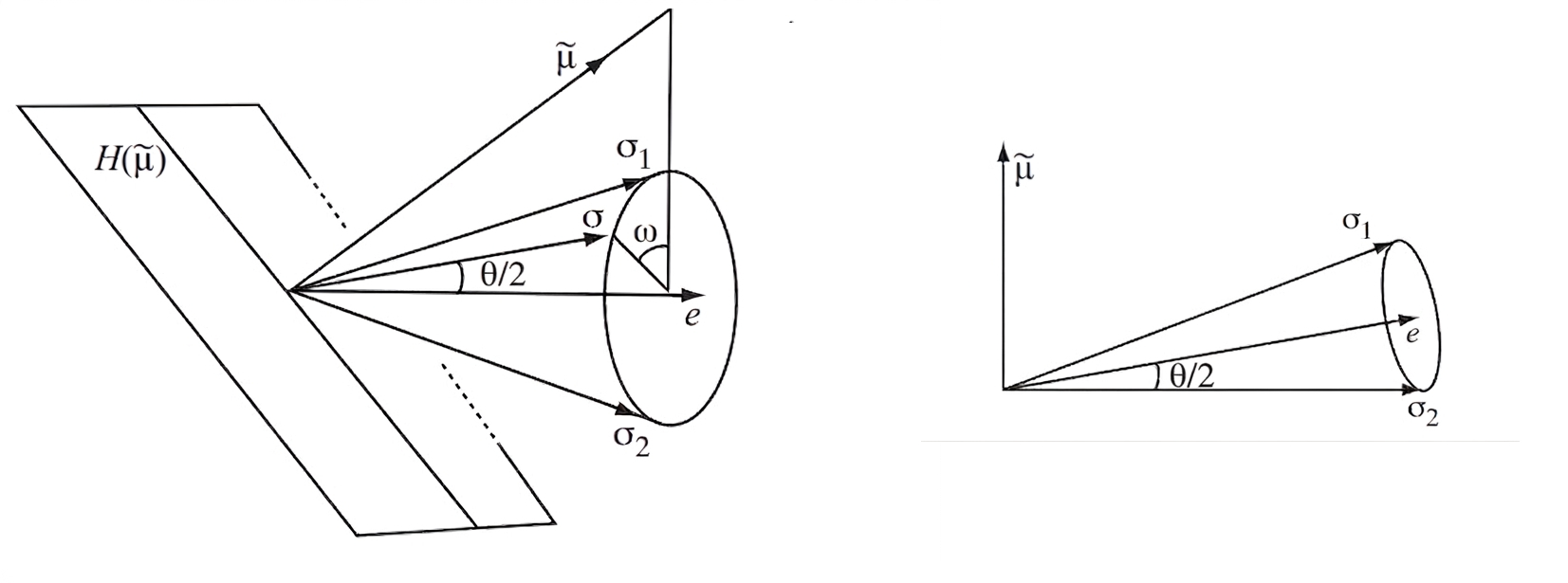}
\caption{Cone of monotonicity (modified from \cite[Figures 4.4 and 4.5]{cbook})}
\end{figure}

Let $\beta:=\frac{\pi}{2}-\theta$ and denote $\gamma:=\langle e,\wu\rangle$. Then $\gamma\leq \beta$ as $W_{\theta,e}\subset H(\wu)$. Let $\sigma_1, \sigma_2$ be two unit vectors such that they are the two generatrices of $W_{\theta/2, {e}}$ belonging to $\text{span}\{\widetilde{\mu}, {e}\}$ (see Figure \ref{fig1}). Suppose that $\sigma_1$ is the nearest to $\widetilde{\mu}$ of the two. Thus, the angles satisfy:
\begin{equation*}
\langle \sigma_1, \widetilde{\mu} \rangle \leq \max\left\{\frac{\pi}{2} - \frac{3\theta}2,\frac{3\theta}2-\frac{\pi}{2} \right\}, \quad \langle \sigma_2, \widetilde{\mu} \rangle \leq \frac{\pi}{2}-\frac{\theta}{2}. 
\end{equation*}
Note that the gain is given by $\eta_\sigma$ and so these two directions give the maximum and the minimum gain in the opening of the cone $W_{\theta, {e}}$, respectively.

Consider a general generatrix $\sigma \in \partial W_{\theta/2, {e}}$, $|\sigma| = 1$, and let $\omega$ be the solid angle (dihedral angle) between the planes $\text{span}\{{e}, \sigma\}$ and $\text{span}\{{e}, \widetilde{\mu}\}$.

We decompose $\sigma$ and $\widetilde\mu$ into components parallel and orthogonal to $e$:
\[
\sigma/|\sigma| = (\cos \tfrac\theta2) e + (\sin\tfrac\theta2)\sigma_\perp, \quad \wu = (\cos \beta) e + (\sin\beta)\wu_\perp
\]
where $\sigma_\perp, \wu_\perp \in e^\perp$ and are unit vectors. The inner product $ \sigma\cdot \wu $ is then:
\beq\lb{4.6}
\begin{aligned}
 \frac{\sigma \cdot \wu}{|\sigma|} &= \left( \left(\cos \frac{\theta}{2}\right) e + \left(\sin \frac{\theta}{2}\right) \sigma_\perp \right) \cdot \left( \left(\cos \beta\right) e + \left(\sin \beta\right) \wu_\perp \right) \\
&= \cos \frac{\theta}{2} \cos \beta + \cos \frac{\theta}{2} \sin \beta \left( e \cdot \wu_\perp \right) + \sin \frac{\theta}{2} \cos \beta \left( \sigma_\perp \cdot e \right) + \sin \frac{\theta}{2} \sin \beta \left( \sigma_\perp \cdot \wu_\perp \right) \\
&= \cos \frac{\theta}{2} \sin\theta + \sin \frac{\theta}{2} \cos \theta \cos \omega,
\end{aligned}
\eeq
where, in the last equality, we used $\cos \omega :=  \sigma_\perp\cdot \wu_\perp$ and $\beta=\tfrac\pi2-\theta$. Then 
\[
\frac{\sigma \cdot \wu}{|\sigma|}\geq \sin\frac\theta2\implies
\langle\sigma,\wu\rangle\in [0,\frac{\pi}2-\frac{\theta}{2}].
\]
Therefore, 
by \eqref{4.6}, trig identities and $\sin \langle\sigma,\wu\rangle\leq \cos\frac\theta2$,
\begin{align*}
\cos\left(\langle\sigma,\wu\rangle + \frac{\theta}{2}\right) &= \frac{1}{2}\sin\theta (1 + \cos \theta (1 + \cos \omega)) - \sin \langle\sigma,\wu\rangle \sin \frac{\theta}{2}\\
    &\geq \frac{1}{2}\sin\theta  \cos \theta (1 + \cos \omega).
\end{align*}
Thus, when $\omega\leq \frac{3}{4}\pi$, then
\[
\eta_\sigma=\cos\left(\langle\sigma,\wu\rangle + \frac{\theta}{2}\right)\geq \frac12\sin\frac\pi4\sin\beta(1-\cos\frac{3}{4}\pi) \geq 0.05\sin\beta
\]
and
\beq\lb{4.66}
\frac\pi2-\langle \sigma,\widetilde{\mu}\rangle-\frac{\theta}{2}\geq \cos\left(\langle\sigma,\wu\rangle + \frac{\theta}{2}\right)\geq 0.05\sin\beta. 
\eeq

Now, let $\mu^1$ be a unit vector pointing to the direction of 
\[
\sigma_1-(\sigma_1\cdot e)e\in \text{span}\{{e}, \widetilde{\mu}\}.
\]
Then $\mu^1\perp {e}$. 
Set $\bar{e} := e+\eps\beta\kappa \mu^1$ for some $\eps\in(0,\frac12)$ to be determined. We will use $\bar\mu$ as the new axis for the larger cone (contained in $S_{\kappa,\alpha}$) that we will define.

For an arbitrary vector $\nu$ such that $\nu\cdot e>0$, we define $\omega_\nu$ to be the solid angle between the planes $\text{span}\{{e}, \nu\}$ and $\text{span}\{{e}, \widetilde{\mu}\}$. It is direct to see that, since $\sigma_1$ is the nearest generatrices of $W_{\theta/2,e}$ to $\wu$, $w_\nu$ really denotes the angle between the projection of $\sigma_1$ to $e^\perp$ and the projection of $\nu$ to $e^\perp$.

Now, consider $\sigma\in \partial W_{\theta/2, {e}}  $ such that $\omega_\sigma\leq \frac{3}{4}\pi$ (see Figure \ref{fig2}). Then, by \eqref{4.66},
\[
E_\sigma=\frac{\pi}{2}-\frac{\theta}2-\langle \sigma,\widetilde{\mu}\rangle\geq 0.05\sin\beta\geq\alpha
\]
if we pick $c_4$ to be small enough.
Then, by the definition of $S_{\kappa,\alpha}$,
\[
\bar{\sigma}: = \sigma + \rho_\sigma\xi\in {S_{\kappa,\alpha}},
\]
where $ \sigma$ is such that $\sigma_\omega\leq\frac{3}{4}\pi$ and $\xi$ is any unit vector.

Since $\langle\sigma, e\rangle=\frac\theta2$ and for a universal constant $\eps_1>0$ (independent of $\beta,\kappa$) we have
\[
\rho_\sigma\geq |\nu|\sin\frac\theta2(1+0.05\kappa\sin\beta)\geq |\nu|\sin(\frac{\theta}{2}+2\eps_1\kappa\beta),
\]
then
\[
W_{\theta+2\eps_1\kappa\beta,e}\bigcap \{\nu:\omega_\nu\leq \frac{3}{4}\pi\}\subseteq S_{\kappa,\alpha}.
\]
Thus, we can fix $\eps$ to be sufficiently small depending only on $\eps_1$ so that
\[
W_{\theta+\eps_1\kappa\beta,\bar e}\bigcap \{\nu:\omega_\nu\leq \frac{3}{4}\pi\}\subseteq S_{\kappa,\alpha}.
\]

\begin{figure}[t]
\label{fig2}
    \centering
    \includegraphics[scale=0.25]{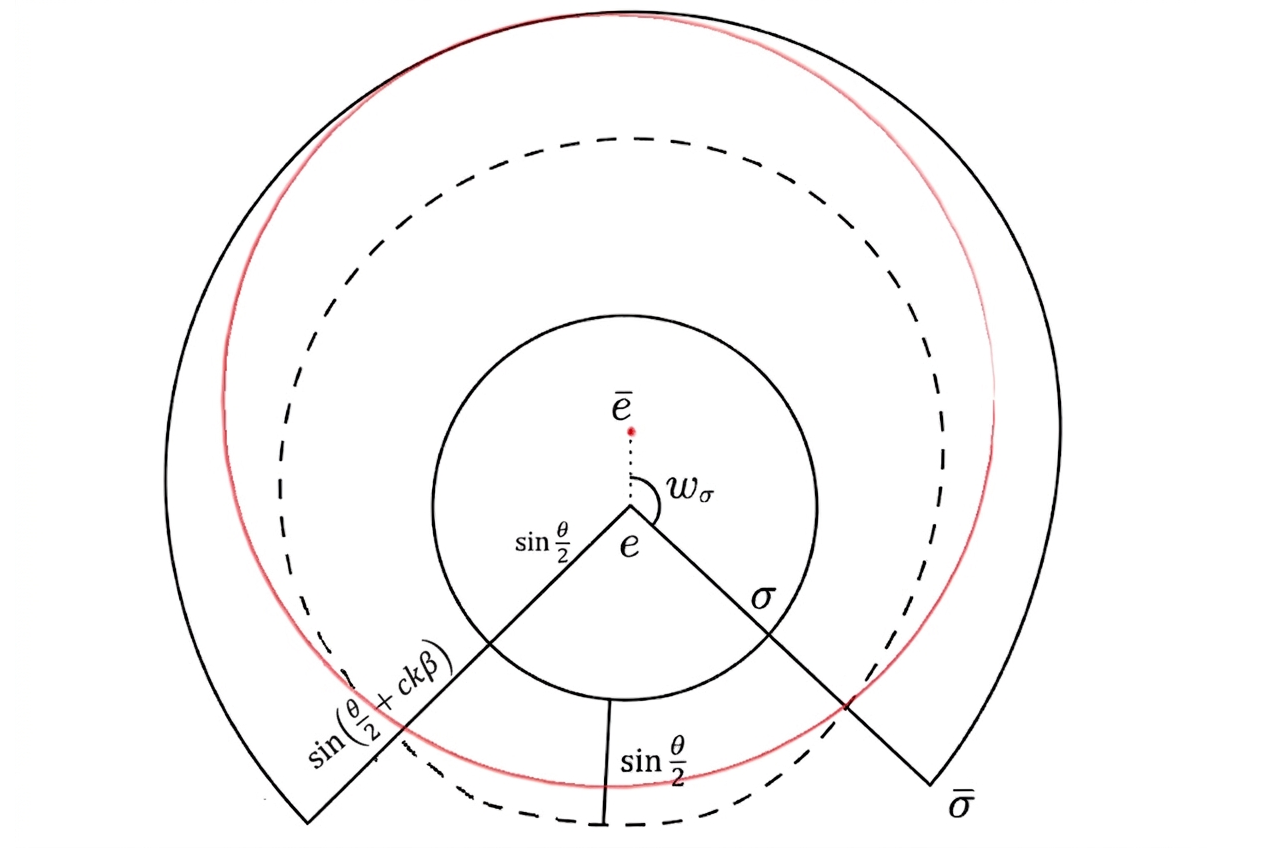}
\caption{The enlarged cone}
\end{figure}

As for the vectors satisfying $\omega_\nu> \frac{3}{4}\pi$, we use that $W_{\theta,e}\subseteq S_{\kappa,\alpha}$. Indeed, for some $\eps_2>0$ to be determined, take any $\nu\in W_{\theta+\eps_2\kappa\beta,\bar e}$ with $\omega_\nu> \frac{3}{4}\pi$. Since the angle between the projection of $\sigma_1$ to $e^\perp$ (which is $\mu_1$) and the projection of $\nu$ to $e^\perp$ is larger than $\frac34\pi$, it follows that
\[
\nu\cdot\mu^1= (\nu-(\nu\cdot e)e)\cdot \mu^1\leq |\nu| \sin(\theta+\eps_2\kappa\beta)\cos\frac34\pi\leq -\frac12|\nu|.
\]
Consequently, if taking $\eps_2$ sufficiently small compared with $\eps$,
\begin{align*}
\nu\cdot  e &=\nu\cdot \bar e+(\eps\beta\kappa)\nu \cdot\mu^1\\
&=(\cos(\theta+\eps_2\kappa\beta)\sqrt{1+\eps^2\beta^2\kappa^2}- \eps\beta\kappa/2 )|\nu|\\
&\leq (\cos(\theta+\eps_2\kappa\beta)+\eps^2\beta^2\kappa^2/2- \eps\beta\kappa/2 )|\nu|\\
&\leq \cos\theta|\nu|.    
\end{align*}
This shows that $\nu\in W_{\theta,e}$ and so $W_{\theta+\eps_2\kappa\beta,\bar e}$ is indeed a subset of $S_{\kappa,\alpha}$ and $\bar\theta:=\theta+\eps_2\kappa\beta$ satisfies
\[ 
\frac{\pi}{2} - \bar{\theta} \leq (1 -\eps_2\kappa) \left( \frac{\pi}{2} - \theta \right). 
\]

\end{proof}

\subsection{Proof of Theorem \ref{T.8.4}}

We will iterate Lemma \ref{L.8.3} to get that the free boundary is $C^1$.

Let $(0,0)\in\Gamma_u$ and let $\widetilde u(x,t):=\tfrac1{\kappa}u( \kappa x+X(\kappa t),\kappa t)$ for some $\kappa>0 $ sufficiently small. Then, we define $\mu_0:=-e_d$, and let $K\geq 2$ from Lemma \ref{Lipdomain}. For each $n\geq 0$,
\[
u_n(x,t):={\gamma_n}{K^{n}}\widetilde u(K^{-n} x, K^{-n}\gamma_n t)
\quad\text{ where }
\]
\[
\gamma_n:=
\begin{cases}
\left[K^n\widetilde u(K^{-n}\tfrac{3}{4}\mu_n,0)\right]^{-1} &\text{ if }\widetilde u(K^{-n}\tfrac{3}{4}\mu_n,0)\leq K^{-n},\\
1 &\text{ otherwise.}
\end{cases}
\]
Here $\mu_n$ is a unit vector which will be defined via iteration. It is direct from the definition that
\[
\gamma_n\leq 1\quad\text{and}\quad
u_n(\tfrac{3}{4}\mu_n,0)\leq 1.
\]
It is also clear by the assumption that $u_0$ is non-decreasing along all directions in $W_{\theta_0,\mu_0}$ with $\theta_0:=\theta$ in $\calQ_1$. We  
set $\tau_0:=1$, $\theta_{-1}:=\theta$.

Recall the constants $c_0,c_2>0$ from Lemma \ref{L.7.4} and $c_4,c_5>0$ from Lemma \ref{L.cone}. Suppose that we have defined $\mu_k,\tau_k,\theta_k$ for $k=0,1\ldots,n$ and they satisfy the following properties:
\begin{enumerate}
    \item $u_k$ is non-decreasing along all directions in $W_{\theta_k,\mu_k}$ in $B_1\times (-\tau_k,\tau_k)$;
\smallskip



\item $\tau_k\geq c_2c_4(\pi-2\min\{\theta_{k-1},\theta_k\})$;

\smallskip

\item    Writing  $\delta_k:=\pi-2\theta_k$, then 
either 
    \[
\delta_{k}\leq CK^{-k+1}\quad\text{ or}\quad         \delta_{k}\leq \delta_{k-1}(1-c_5\delta_{k-1})
    \]
for some $C>0$ independent of $k$.
\end{enumerate}
The goal is to show (1)--(4) for $k=n+1$. 

\smallskip

First of all, let
\[
\alpha:=c_4 (\pi-2\theta_n)
\]
If the opening $\theta_n$ is already large such that $c_2\alpha/2\leq C\kappa K^{-n}$, then we simply define
\[
\mu_{n+1}:=\mu_n,\quad \theta_{n+1}:=\theta_n\quad\text{and}\quad \tau_{n+1}:=\tau_{n}.
\]
Note that
\[
u_{n+1}=\frac{K\gamma_n}{\gamma_{n+1}}u_n(\frac{x}{K},\frac{\gamma_n}{K\gamma_{n+1}}t),
\]
and $\frac{\gamma_n}{K\gamma_{n+1}}\leq 1$ by Lemma \ref{Lipdomain}. Thus, the conditions (1)--(3) hold the same.

Without loss of generality, we can assume $c_2\alpha/2\geq C\kappa K^{-n}$. 
We define
\[
\widetilde{\mu}:=\nabla u_n(\tfrac{3}{4}\mu_n,0)/|\nabla u_n(\tfrac{3}{4}\mu_n,0)|\quad\text{and}\quad \tau:=c_2\alpha.
\]
Then $\tau= c_2c_4(\pi-2\theta_n)\leq \tau_n$.
For any $\nu\in W_{\theta_n/2,\mu_n}$ satisfying $|\nu|<c_0$, set
\[
\delta_\nu:=|\nu|\sin(\theta_n/2)\quad\text{and}\quad
\eta_\nu :=c_2\cos(\langle \widetilde{\mu},\nu\rangle+\tfrac{\theta_n}2).
\]
Consequently, if $\nu$ is such that $\alpha\leq \tfrac{\pi}{2}-\tfrac{\theta_n}{2}-\langle \widetilde{\mu},\nu\rangle$, then
\[
\eta_\nu\geq c_2\cos(\tfrac\pi2-\alpha)\geq c_2\alpha/2.
\]

Recall $c_2\alpha/2\geq C\kappa K^{-n}$ and \eqref{nondegn}. We have
\[
\eta_\nu\geq C\kappa K^{-n},\quad  u_n(\tfrac{3}{4}\mu,0)\leq 1,
\]
and the non-degeneracy property of $u_n$ holds easily by that of $u$.
Therefore, if
\beq\lb{4.77}
\alpha\leq \tfrac{\pi}{2}-\tfrac{\theta_n}{2}-\langle \widetilde{\mu},\nu\rangle,
\eeq
it follows from Lemma \ref{L.8.3} that for some $c_3>0$, 
\[
\sup_{y\in B_{(1+c_3\alpha\eta_\nu\gamma)\delta_\nu}(x,t)}u_n(y-\nu,t) \leq u_n(x,t)\quad\text{ in } B_{1/2}\times [-\tau/2,\tau].
\]
Then $u_n$ is non-decreasing along all the directions of
\[
S_{n,\alpha}^{\nu}:=\bigcup_{\nu\in W_{\theta_n/2,\mu_n}}B_{(1+c_3\gamma\alpha\eta_\nu}(\nu)\quad \text{ in }B_{1/2}\times [-\tau/2,\tau],
\]
for all $\nu\in W_{\theta_n/2,\mu_n}$ satisfying \eqref{4.77}.
Hence, $u_{n+1}$ is non-decreasing along all the directions of $S_{n,\alpha}^{\nu}$ with $\nu\in W_{\theta_n/2,\mu_n}$ satisfying \eqref{4.77} for 
\[
(x,t)\in B_{K/2}\times [-\tfrac{K\gamma_{n+1}}{2\gamma_{n}}\tau,\tfrac{K\gamma_{n+1}}{2\gamma_{n}}\tau]\subseteq B_{1}\times [-\tau,\tau],
\]
where we used Lemma \ref{Lipdomain} in the inclusion.

So, we can  apply Lemma \ref{L.cone} to get that  there exists $\mu_{n+1}$ and $\theta_{n+1}>\theta_n$ such that in $B_{1}\times [-\tau,\tau]$,
\begin{enumerate}
    \item $u_{n+1}$\text{ is non-decreasing along all directions  of $W_{\theta_{n+1},\mu_{n+1}}$};
    \smallskip
    \item There exists $c_5>0$ such that for $\delta_n=\pi-2\theta_{n}$,
    \[
    \delta_{n+1}\leq \delta_n(1-c_5\delta_n).
    \]
\end{enumerate}
Finally, we take $\tau_{n+1}:=\tau$ and then
\[
\tau_{n+1}=c_2c_4(\pi-2\theta_n)\geq c_2c_4(\pi-2\theta_{n+1}).
\]
Hence, we conclude (1)--(3) with $k$ replaced by $n+1$.

\smallskip

By induction, we conclude that $u$ is non-decreasing along all directions in $W_{\theta_n,\mu_n}$ in $B_{K^{-n}}\times (-K^{-n}\gamma_n\tau_n,K^{-n}\gamma_n\tau_n)$. By an argument of induction and (3), we obtain that
\[
\delta_n\leq \tfrac{C}{n}\quad\text{ for some }C>0.
\]
This shows that the free boundary $\Gamma_u(0)$ is $C^1$ at $0$.

Moreover, if $\Gamma_u(0)$ is given by a graph $F(x')=x_d$ locally near the origin, then $F$ is differentiable and for $(x',F(x')),(y',F(y'))\in B_{K^{-n}}$, we have
\[
|\nabla F(x')-\nabla F(y')|\leq C/n.
\]
This yields a continuity regularity of $F$:
\[
|\nabla F(x')-\nabla F(y')|\leq C (-\log_K |x'-y'|)^{-1}.
\]


\section{Periodic vertical advection Hele-Shaw}\lb{S5}


Consider the periodic domain $(x,y,t)\in \bbT\times \bbR\times [0,\infty)$, and set the vector field $b=(0,\Theta(x))$ for some periodic function $\Theta$. Then $b$ is divergence free and the equation \eqref{1.1} becomes
\begin{equation}\lb{2.6}
\begin{cases}
\Delta u = 0 \quad \text{ in } \{u>0\}, \\
\partial_t u=|\nabla u|^2+\Theta(x) \partial_yu \quad \text{ on }\partial\{u>0\}.
\end{cases}
\end{equation}

We extend $u$ and $\Theta$ periodically to all $x \in \mathbb{R}$. In what follows, we shall identify these functions with their periodic extensions without further comment.

\subsection{Equation of the free boundary}
We study a class of front propagation solutions $u$ such that for some $c_*>0$,
$u(x,y,t)$ behaves like $-c_*y$  as $y\to -\infty$,
and the free boundary $\Gamma_u(t)$ is given by a graph $y=F(x,t)$.

Let us start with introducing the equation of the free boundary.
Suppose that the ``fluid domain'' is given by 
\[
\Omega_{F,t}:=\{(x,y),\, y<F(x,t)\}.
\]
For each $t\geq 0$, let $\phi^t(x,y)$ satisfy
\begin{equation}\lb{a.1.1}
\begin{cases}
\Delta\phi^t = 0 \quad in \ \Omega_{F,t}, \\
\phi^t(x, F(x,t)) =F(x,t), \quad \nabla\phi^t \in L^2(\Omega_{F,t}).
\end{cases}
\end{equation}
Then $u(x,y,t):=c_*(\phi^t(x,y)- y)_+$ is expected to be the solution to \eqref{2.6} if the free boundary $F$ evolves ``correctly''.
It was proved in  \cite[Proposition 3.6]{np20} and \cite[Proposition 2.6]{dong21} that \eqref{a.1.1} has a unique variational solution when $F(\cdot,t) \in \text{Lip}(\mathbb{T})$, and
\begin{equation*}
\|\phi^t\|_{\dot{H}^1(\Omega_{F,t})} \le C\left(1 + \|F(\cdot,t)\|_{\text{Lip}(\mathbb{T}^d)}\right)\|F(\cdot,t)\|_{\dot{H}^{\frac{1}{2}}(\mathbb{T}^d)}.
\end{equation*}
Set $N^t(x):=(-F_x(x,t),1)$.
Following \cite{dong21}, we define
\beq\lb{calG}
\calG (F)(x,t):=\partial_{N^t}\phi^t(x,F):= \lim_{h \to 0^-} \frac{1}{h} \left[ \phi^t((x, F) + hN^t(x)) - \phi^t(x, F) \right].
\eeq
In particular, $\calG (F+C)=\calG (F)$ and 
if $F(\cdot,t)$ is constant, then $\calG (F)(\cdot,t)\equiv 0$.

It follows from \eqref{2.6} that the boundary velocity to the outer normal direction (and multiplied by $|N^t|$) is given by
\begin{align*}
N^t(x)\cdot V(x,F(x,t))&=-N^t(x)\cdot ((\nabla \phi^t)(x,F(x,t))-(0,c_*))-N^t(x)\cdot (0,\Theta(x))\\
&=-c_*\calG (F(\cdot,t))(x)+c_*-\Theta(x).    
\end{align*}
Since $\partial_t F(x,t)=N^t(x)\cdot V(x,F(x,t))$, we obtain the equation for the free boundary variable as follows
\beq\lb{main}
\partial_t F(x,t)=-{c_*}\calG (F(\cdot,t))(x)+c_*-\Theta (x).
\eeq

It was proved in \cite{dong21} and \cite{dong23} the global well-posedness and Lipschitz regularity of
$F_t=- {c_*}\calG (F)$
in dimension $2$ and $3$, respectively. The global well-posedness in general dimension was proved in \cite{schwab2024well}.

\subsection{Viscosity solutions of the boundary variable}

The following definition is given in \cite[Definition 6.1]{dong21}.
\begin{definition}{\rm [Viscosity solutions  of the boundary variable]}\lb{D.5.1}
A function $F : \mathbb{T} \times [0,T]\to\bbR$ is called a viscosity subsolution (resp. supersolution) of \eqref{main} if
\begin{itemize}
    \item[\textit{(i)}] $F$ is upper semicontinuous (resp. lower semicontinuous) on $\mathbb{T} \times [0,T]$;
    \smallskip
    \item[\textit{(ii)}] for every $\psi : \mathbb{T} \times (0,T) \to \mathbb{R}$ with $\partial_t\psi \in C(\mathbb{T} \times (0,T))$ and $\psi \in C((0,T); C^{1,1}(\mathbb{T}))$, if $F - \psi$ attains a global maximum (resp. minimum) over $\mathbb{T} \times [t_0 - r, t_0]$ at $(x_0, t_0) \in \mathbb{T} \times (0,T)$ for some $r > 0$, then
    \begin{equation}\lb{a.1.2}
 \partial_t\psi(x_0, t_0)\leq -{c_*}\calG (\psi)(x_0)+c_*- \Theta(x_0)  \quad (\textit{resp. } \ge).
\end{equation}
\end{itemize}
A viscosity solution is both a viscosity subsolution and viscosity supersolution.
\end{definition}

We next show that, under certain conditions,  the notion of solutions given in Definition \ref{D.5.1} is consistent with the viscosity solutions given in Definition \ref{D.21}.

\begin{lemma}{\rm [Consistency]}\lb{L.5.2}
Let $F : \mathbb{T} \times [0,T]\to\bbR$   be a viscosity solution to \eqref{main} and let $u(x,y,t):=(\phi^t(x,y)-c_* y)_+$ with $\phi^t$ given by \eqref{a.1.1}. Then if either $F_t\leq C$ for some $C>0$ or $u$ is non-degenerate, we have $u$ is a viscosity solution to \eqref{2.6}. In particular, $\Omega_{F,t}=\Omega_u(t)$.
\end{lemma}
\begin{proof}
By the above definition of $u$, $u(\cdot,t)$ is harmonic and $u$ is continuous, $u(\cdot,t)>0$ in $\Omega_{F,t}$, and $u(\cdot,t)=0$ on $\partial \Omega_{F,t}$. It remains to verify the boundary condition. We will only prove that $u$ is a viscosity subsolution and the fact that $u$ is a viscosity supersolution follows similarly.

Let $\varphi\in C^{2,1}_{x,t}$ such that $u-\varphi$ has a local maximum at $(x_0,y_0,t_0)$ in $\bigcup_{t\leq t_0}\Omega_{F,t}\times\{t\}$, and $x_0\in \Gamma_{F,t_0}$ and $\varphi(x_0,y_0,t_0)=0$.  
In the case when $u$ is non-degenerate, then $\varphi$ is non-degenerate and so, in particular, $\nabla\varphi(x_0,y_0,t_0)\neq 0$.

\medskip

\noindent{\bf Step 1.} In this step, we assume that $\nabla\varphi(x_0,y_0,t_0)\neq 0$. To show that $u$ is a viscosity subsolution, it suffices to show 
\beq\lb{a.1.3}
V_0=\varphi_t/|\nabla\varphi|(x_0,y_0,t_0)\leq (|\nabla\varphi|+b\cdot\nabla\varphi/|\nabla\varphi|)(x_0,y_0,t_0).
\eeq
assuming this condition and that $-\Delta\varphi(x_0,y_0,t_0)>0$.
By modifying $\varphi$ outside a neighborhood of $(x_0,y_0,t_0)$, we can also assume that $u\leq \varphi$ globally for $t\leq t_0$ and $\varphi$ is superharmonic globally.

By the implicit function theorem, the $0$-level set of $\varphi$ is a $C^2$ curve near $(x_0,y_0,t_0)$ for each $t$ and is $C^1$ in time. Then $V_0$ is the same as the boundary velocity of the level set at $(x_0,y_0,t_0)$. Moreover, we can find a global $C^2$ function $G(x,t)$ such that $y=G(x,t)$ agrees with the level set near  $(x_0,y_0,t_0)$ and $G$ touches $F$ from above at the point. Thus, \eqref{a.1.2} yields
\[
 \partial_tG(x_0, t_0)\leq -{c_*}\calG (G)(x_0)+c_*- \Theta(x_0).
\]
Since $G$ is $C^2$, by the definition of $\calG$, this inequality says that the velocity of $y=G(x,t)$ at $(x_0,t_0)$ is no more than 
\begin{align*}
&-(N^{t_0}(x_0)/|N^{t_0}(x_0)|)\cdot (\nabla \varphi_1(x_0,y_0,t_0)+ (0,\Theta(x_0))\\
&\qquad =    |\nabla \varphi_1(x_0,y_0,t_0)|+\Theta(x_0) (\varphi_1)_y(x_0,y_0,t_0)/|\nabla \varphi_1(x_0,y_0,t_0)|,
\end{align*}
where $\varphi_1$ satisfies \eqref{a.1.1} with $F$ replaced by $G$. Note that $\nabla\varphi/|\nabla\varphi|=\nabla\varphi_1/|\nabla\varphi_1|$ which is equal to the inward normal direction, and since $\varphi\geq u$ is superharmonic and $\varphi_1$ is harmonic, we get $|\nabla\varphi_1|\leq |\nabla\varphi|$ at $(x_0,y_0,t_0)$. Hence, \eqref{a.1.3} is proved.

\medskip

\noindent{\bf Step 2.} Because of the previous arguments, it remains to show
\[
\varphi_t(x_0,y_0,t_0)\leq (|\nabla\varphi|^2+b\cdot\nabla\varphi)(x_0,y_0,t_0),
\]
assuming $F_t\leq C$ and $\nabla\varphi(x_0,y_0,t_0)=0$.

Assume for contradiction that $\varphi_t(x_0,y_0,t_0)>0$. Since $u\leq (\varphi)_+$ in $B_r(x_0,y_0)\times [t_0-r,t_0]$ for some $r>0$ sufficiently small and ${\nabla\varphi}(x_0,y_0,t_0)= 0$, there exists $\rho_s>0$ satisfying $\lim_{s\to 0}\rho_s/s=\infty$ such that
\[
u(x,y,t_0-s)\leq (\varphi(x,y,t_0-s))_+= 0
\]
for each $s\in [0,r]$ and $(x,y)\in B_{\rho_s}(x_0,y_0)$.
However, note that this implies that $F(x_0,t_0-s)\leq \rho_s$, while $F(x_0,t_0)=y_0$. Hence, this and $\lim_{s\to 0}\rho_s/s=\infty$ contradict with the assumption that $F_t\leq C$ and so we proved the claim.
\end{proof}

The following result is the comparison principle for the equation of boundary variables. The proof is almost identical to the one of \cite[Theorem 6.3]{dong21}, with a few adaptions.

\begin{proposition}{\rm[Comparison principle]}
Assume that $F, G : \mathbb{T}\times [0,T] \to \mathbb{R}$ are respectively a bounded viscosity subsolution and supersolution of \eqref{main}. If $F(x,0) \le G(x,0)$ for all $x \in \mathbb{T}$, then $F(x,t) \le G(x,t)$ for all $(x,t) \in \mathbb{T} \times [0,T]$.
\end{proposition}

The Perron's method yields the existence of a unique bounded solution to \eqref{main}. Due to the appearance of the drift term, we are only able to obtain local Lipschitz continuity of $F$ in space.
\begin{theorem}{\rm[Wellposedness]}\lb{T.2.5}
Let $F_0,\Theta:\mathbb{T}\to\bbR$ be Lipschitz continuous. For all $T>0$, there exists
\[
F \in C(\mathbb{T} \times [0, T]) \cap L^\infty([0, T]; W^{1,\infty}(\mathbb{T}))
\]
such that $F$  is the unique viscosity solution to \eqref{main} with initial data $F_0$, and
\[
-\|F_0\|_\infty+\min_{x'\in\bbT}\{c_*-\Theta(x')\} T\leq F(x,t) \le \|F_0\|_\infty+\max_{x'\in\bbT}\{c_*-\Theta(x')\}T,
\]
\[
\|F(\cdot,t)\|_{\Lip}\le \|F_0\|_{\Lip}+\|\Theta\|_{\Lip}T.
\]

Furthermore, if
\beq\lb{m1}
\begin{aligned}
m_1:=\min_{x\in \bbT}\{-{c_*}\calG(F_0)(x) + c_* - \Theta(x)\}\in \bbR,
\end{aligned}
\eeq
then $F$ is Lipschitz continuous in $t$ and $F_t(x,t)\in [m_1, M_1]$ for all $(x,t)\in\bbT\times [0,T]$, where
\beq\lb{M1}
M_1:=\max_{x\in \bbT}\{-{c_*}\calG(F_0)(x) + c_* - \Theta(x)\}.
\eeq
\end{theorem}

As a consequence of this result and Lemma \ref{L.5.2}, we will not distinguish the two notions of the solutions.


\begin{proof}
The existence is given by taking supremum of all subsolutions to \eqref{main} with initial data $\leq F_0$. It is direct to see that $F_0(x)+Ct$ and $F_0(x)-Ct$ for some $C>0$ are, respectively, a viscosity super- and a viscosity sub- solution to \eqref{main}. Therefore, the supremum of all subsolutions is equal to $F_0$ at $t=0$. The uniqueness is given by the comparison principle.

Next, let us define
\begin{align*}
&\psi_+(x,t):=\|F_0\|_\infty+\max_{x'\in\bbT}\{c_*-\Theta(x')\}\, t\\
&\psi_-(x,t):=-\|F_0\|_\infty+\min_{x'\in\bbT}\{c_*-\Theta(x')\}\, t.    
\end{align*}
It is clear that 
\[
\partial_t \psi_+ \ge -{c_*}\calG(\psi_+)+c_*-\Theta\quad\text{and}\quad \partial_t \psi_- \le -{c_*}\calG(\psi_-)+c_*-\Theta.
\]
By the comparison principle, we know that the solution $F$ satisfies
\[
\psi_-\leq F\leq \psi_+.
\]

To see that $F$ is Lipschitz continuous in space, we let $z\in (-1,1)$. 
Direct computation yields that $F(\cdot+z,t)$ is a viscosity solution to
\[
\partial_t F(x+z,t) = -{c_*}\calG(F(\cdot,t))(x+z)+c_*-\Theta(x+z).
\]
Therefore, 
\[
G(x,t):=F(x+z,t)+(\|F_0\|_{\Lip}+\|\Theta\|_{\Lip}t)|z|
\]
satisfies
\[
\partial_t G(x,t) = -{c_*}\calG(G(\cdot,t))(x)+c_*-\Theta(x+z)+\|\Theta\|_{\Lip}|z|,
\]
and so it is a viscosity supersolution to \eqref{main}. Since $G(x,0)\geq F(x,0)$, the comparison principle yields that
\[
F-F(\cdot+z,\cdot)\leq (\|F_0\|_{\Lip}+\|\Theta\|_{\Lip}T)|z|\quad\text{ in }\bbT\times [0,T].
\]

Now, assume \eqref{m1}. Let 
\[
\varphi_{m_1}(x,t):=F_0(x)+{m_1} t\quad\text{and}\quad\varphi_{M_1}(x,t):=F_0(x)+{M_1}t,
\]
and then
\[
\partial_t \varphi_{m_1}+{c_*}\calG(\varphi_{m_1})-c_*+\Theta= {m_1}+{c_*}\calG(F_0)-c_*+\Theta\leq 0,
\]
\[
\partial_t \varphi_{M_1}+{c_*}\calG(\varphi_{M_1})-c_*+\Theta= {M_1}+{c_*}\calG(F_0)-c_*+\Theta\geq 0.
\]
The comparison principle again yields 
\[
F_0(x)+{M_1}t\geq F(x,t)\geq F_0(x)+{m_1} t.
\] 
Thus, for any $s>0$, $F_1(x,t):=F(x,t+s)-{m_1} s$ satisfies that $F_1(x,0)\geq F(x,0)$. Since both $F_1$ and $F$ are solutions to \eqref{main} due to the translation invariant property of the operator, by the comparison principle again, we get
\[
F(x,t+s)\geq F(x,t)+{m_1} s.
\]
Similarly, by comparing $F_2(x,t):=F(x,t+s)-{M_1}s$ and $F(x,t)$, we can obtain $F(x,t+s)\leq F(x,t)+{M_1} s$. These yield the second conclusion.
\end{proof}


\subsection{Monotonicity properties}
Let $F$ be a solution to \eqref{main} and then $u$ be defined from Lemma \ref{L.5.2}.
We first show that $u$ is monotone in its support.

\begin{lemma}\lb{L.5.5}
Let $F(\cdot,t)$ be Lipschitz continuous with Lipschitz constant $c_L$. We have 
\[
\nabla_e u(\cdot,t)\leq 0\quad\text{for all $e=(e_1,e_2)$ such that $|e_1|\leq c_L e_2$}.
\]
\end{lemma}
\begin{proof}
We fix $t$ and we drop it from the notations of $F$ and $u$. Writing $\vec z=(x,y)$, it suffices to show that for any $\eps>0$, $u(\vec z+\eps e)\geq u(\vec z)$. Indeed, since $F(\cdot)$ is Lipschitz continuous with Lipschitz constant $c_L$ and $u$ is positive below the graph of $y=F(x)$, then $u(\vec z+\eps e)\geq u(\vec z)$ on the graph. On the other hand, note that $u(x,y)=\phi^t(x,y)-c_*y$ where $\phi^t$ solves \eqref{a.1.1}. 
So $u((x,y)+\eps e)\geq u((x,y))$ as $y\to -\infty$. Hence, the comparison principle yields that $u(\vec z+\eps e)\geq u(\vec z)$ for all $\vec z$, which finishes the proof.
\end{proof}

If we further assume $m_1$ from \eqref{m1} to be large, then $u$ is strictly increasing along streamlines. 
In our case of $b(x,y)=(0,-\Theta(x))$, the streamline starting at $(x,y)\in\bbT\times\bbR$ is given by
\beq\lb{def.X}
X(t;x,y)=(x,y-\Theta(x)t).
\eeq
The notations $e_1$ and $e_2$ stand for the positive $x$ and $y$ directions, respectively. 

\begin{lemma}\lb{L.2.6}
Let ${m_1}$ from \eqref{m1} and further assume that
\beq\lb{detla1}
{m_1}+\min_{x\in\bbT}\Theta(x)\geq 0.
\eeq
Then for any $\eps>0$ and ${\delta}\in [0,{m_1}+\min_{x\in\bbT}\Theta(x)]$, we have
\[
\begin{aligned}
u(X(\eps;x,y)+{\delta} \eps e_2,t+\eps)
\geq u(x,y,t).
\end{aligned}
\]
\end{lemma}
\begin{proof}
Because $F_t\geq {m_1}$, for all $\eps>0$, we have
\[
F(x,t+\eps)\geq F(x,t)+{m_1}\eps.
\]
So, by the comparison principle for harmonic functions, 
\[
u(x,y+{m_1} \eps,t+\eps)\geq u(x,y,t).
\]
Since $\delta-\Theta(x)\leq m_1$ and $u_y\leq 0$ by Lemma \ref{L.5.5}, we get 
\[
u(x,y+({\delta}-\Theta(x)) \eps,t+\eps)\geq u(x,y,t),
\]
which yields the conclusion.
\end{proof}


\subsection{Non-degeneracy}
In this section, we show non-degeneracy of the pressure variable. The results are essentially proved in \cite{CJK} for the case ${b}=0$ and \cite{kim2} for the general case. 
We will use the notation 
\[
\calQ_T:= \bbT\times\bbR\times (-T,T)\quad\text{ for some }T>0,
\]
and $M_1$ from \eqref{M1}.

\begin{lemma}\lb{C.6.4}
Let $u$ solve \eqref{1.1} in $\calQ_T$ for some $T\geq 1$, and let the free boundary be given by $F$. Assume that $F_0(0)=0$,  $ M_1 \asymp c_*$ and $c_*\geq 1$, and that
$u(\cdot,\cdot,t)$ is non-decreasing with respect to $W_{\theta,-e_2}$ for some $\theta \in (\frac\pi3, \frac\pi2)$
Then, there exists $a>0$ independent of $c_*$ such that in $\calQ_{T/2}$ we have
\[
-\partial_y u(x,y,t)\geq a\, c_*.
\]
\end{lemma}

Here $\theta>\frac\pi3$ is needed so that we can apply the results in \cite[Theorem 7.5]{kimzhang2024}.

\begin{proof}
Since $u$ is non-decreasing with respect to $W_{\theta,-e_2}$, we have $F_0\in( -\frac12,\frac12)$.
Then by Theorem \ref{T.2.5}, it follows that 
\[
-M_1T-1 \leq F(x,-T) \leq F(x,t) \leq F(x,T) \leq M_1T+1
\]
for all $t\in (-T,T)$. The comparison principle for harmonic functions yields
\[
 c_*(c_*(t+T)-y-M_1T-1)_+\leq u(x,y,t)\leq  c_*(c_*(t+T)-y+M_1T+1)_+.
\]
Hence 
\[
 c_*(c_*T+M_1T+1)\leq u(x,-c_*T-2M_1T-2,t)\leq  3c_*(c_*T+M_1T+1).
\]

Now, let $v(x,y,t):=c_*^{-2}u(c_*x, c_*y,t)$, which then satisfies
\[
\left\{
\begin{aligned}
-\Delta v& =0 && \text{ in }\{ x<\tilde F( y,t)\},\\
\partial_tv&=|\nabla v|^2+\widetilde\Theta( x) \partial_yv  &&\text{ on }\partial\{x=\tilde F(y,t)\},
\end{aligned}
\right.
\]
where $\widetilde F(x,t):=c_*^{-1}F(c_* x,t)$ and $\widetilde\Theta( x):=c_*^{-1}\Theta(c_* x)$. From the above,  for $y_0:=-T-\frac{2M_1T+2}{c_*}$ and for all $t\in (-T,T)$,
\beq\lb{3.2}
 T+\frac{M_1T+1}{c_*}\leq v(x,y_0,-T)\leq  3\left(T+\frac{M_1T+1}{c_*}\right).
\eeq

To prove the lemma, it suffices to show that
\[
-\partial_y \widetilde u(x,y,t)\geq a \quad\text{ for }t\in (-T/2,T/2).
\]
We verify the assumptions of \cite[Theorem 7.5]{kimzhang2024}.
Indeed, (H-a') follows from (1) and $\theta_\beta={\pi}/{(2\beta)}$ in space dimension 2 by the remark after \cite[Lemma 2.10]{kimzhang2024}. 
(H-b) is due to \eqref{3.2}.
Note that the vector field $\vec b$ in \cite{kim2} is equal to $b(x,y)=(0,\Theta(x))$ in our setting. Thus, the condition (H-c) follows from Lemma \ref{L.2.6}. The last condition also follows from \eqref{3.2}.

Now, we can apply \cite[Theorem 7.5]{kimzhang2024} to get that for some $\eps_0,c_0>0$ and for all $\eps\in (0,\eps_0)$,
\begin{equation*}
     v(x,y-\eps,t)\geq c_0\eps \quad \hbox{ for all } (x,y,t)\in \Gamma_v\cap \calQ_{T/2}.
\end{equation*} 
By \cite[Lemma 11.11]{cbook}, if $\widetilde F(x,t)-y>0$ is sufficiently small and $(x,y)\in\Omega_v(t)$, then there exists $c\in (0,1)$ such that
\[
-\partial_y v(x,y,t)\geq C^{-1}v(x,y,t)/\dist((x,y),\{y=F(x,t)\})\geq c,
\]
where the last inequality is due to the Lipschitz continuity of the free boundary $y=\widetilde F(x,t)$. While in the interior of $y<\widetilde F(x,t)$, since $v$ is harmonic and $v\sim -y$ as $y\to-\infty$, applying the maximum principle to $-\partial_y v$ yields the conclusion.
\end{proof}

\subsection{Smooth free boundary}
In this subsection, we proceed to show the main result for the periodic vertical Hele-Shaw flow. In this setting, the interior gain happens automatically if one stays far away from the free boundary.

\begin{lemma}{\rm [Interior gain]}\lb{L.6.1}
Suppose that $-\partial_yu(x,y,t)\geq ac_*$ in $\Omega_{u}(t)$. Given any $\theta_0\in (\frac\pi4,\frac{\pi}{2})$, there exist $r_0\geq1,c_2>0$ independent of $c_*$ such that the following holds for all $r\geq r_0$. In the domain of
\[
\Omega_{u}^r:=\left\{(x,y,t)\in\Omega_u\,| d((x,y),\Gamma_u(t))\geq r\right\},
\]
$u(x,y,t)$ is non-decreasing along all directions in $W_{\theta_0,-e_2}$.
Moreover, inside $\Omega_u^r$, for any $\nu\in W_{\theta_0/2,-e_2}$  with $|\nu|\in (0,\frac18)$ we have
\beq\lb{7.12'}
\sup_{B_{(1+c_2\eta_\nu ){\delta_\nu}}(x,y)}u(\cdot-\nu,t)\leq u(x,y,t)-c_*\eta_\nu|\nu|
\eeq
where 
\[
\eta_\nu:=\frac{a}{4}\cos(\langle -e_2,\nu\rangle+\tfrac{\theta_0}2)
\quad\text{and}\quad \delta_\nu:=|\nu|\sin(\theta_0/2) \asymp |\nu|.
\]
\end{lemma}
\begin{proof}
Recall that $\phi^t$ is a solution to \eqref{a.1.1}. The comparison principle yields $|\phi^t(x)|\leq c_*\|F(\cdot,t)\|_\infty$. 
For any $(x,y,t)\in \Omega_u^r$
(then $B((x,y),r)\subseteq \Omega_u(t)$), since $u=(\phi^t-c_*y)_+$, the classical result for Harmonic functions (\cite[Theorem 7, Section 2.2]{evans}) yields
\[
|u_x(x,y,t)|\leq |\nabla \phi^t(x,y)|\leq 16r^{-1}c_*\|F(\cdot,t)\|_\infty.
\]
Since $-u_y\geq ac_*$, we get $\nabla_\sigma u(x,y,t)\geq 0$ for all $\sigma$ such that 
\beq\lb{5.13}
\tan\langle\sigma,-e_2\rangle \leq \frac{ar}{16\|F(\cdot,t)\|_\infty}.
\eeq

Let us denote  
\[
\theta_1:=\arctan \frac{ar}{16\|F(\cdot,t)\|_\infty}\quad\text{and}\quad L:= \frac{ar}{16\|F(\cdot,t)\|_\infty\tan\theta_0}.
\]
By taking $r$ to be sufficiently large, we get $L\geq 2$, $\theta_1\geq \theta_0$ and $\theta_1$ can be arbitrarily close to $\pi/2$.  
In particular, in $\Omega^r_u$, we have $|\nabla u|\leq Cc_*$.

Let $\nu\in W_{\theta_0/2,-e_2}$ with $|\nu|\in (0,\frac18)$.
We fix $t$ and take any $\vec{x}\in \Omega_u^r(t)$ and $\vec{y}\in B_{\delta_\nu}(\vec{x})\cap\Omega_u^r(t)$. Define $\bar\nu:=\nu-(\vec y-\vec x)$. 
As a consequence of the definition, we have $\bar\nu\in W_{\theta_0,-e_2}$ (or $\langle \bar\nu,-e_2\rangle\leq \theta_0$) and so $\nabla_{\bar\nu}u\geq 0$. 
Writing $\mu_{\vec x}:=\nabla u_r(\vec x,0)$, by \eqref{5.13}, we get $\langle -e_2 , \mu_{\vec x}\rangle\leq \frac{\pi}{2}-\theta_1$.
We claim that
\[
\frac{\pi}{2}-\langle \mu_{\vec x},\bar\nu\rangle\geq \frac12(\frac\pi2-\langle -e_2,\bar\nu\rangle).
\]
Indeed, by direct computation, we have
\[
\langle \mu_{\vec x},\bar\nu\rangle\leq \langle -e_2,\bar\nu\rangle+\frac{\pi}{2}-\theta_1\leq \frac12\langle -e_2,\bar\nu\rangle+\frac{\pi}{4}
\]
when $\theta_1\geq \frac{\pi}{4}+\frac{\theta_0}{2}$. Consequently, it follows from the claim that
\begin{align*}
    \nabla_{\bar\nu} u(\vec x,0)  &= |\bar\nu|\cos\langle \mu_{\vec x},\bar\nu\rangle|\nabla  u(\vec x,0)|\geq \frac12 |\bar\nu|\cos\langle -e_2,\bar\nu\rangle|\nabla  u(\vec x,0)|\\
&\geq \frac{ac_*}{2}|\bar\nu|\cos\langle-e_2,\bar\nu\rangle ,
\end{align*}
where, in the second inequality, we used the non-degeneracy property of $u$. This implies that
\[
\begin{aligned}
u(\vec y-\nu,0)&=u(\vec x-\bar\nu,0)\leq u(\vec x,0)-\frac{ac_*}{2}|\bar\nu|\cos\langle-e_2,\bar\nu\rangle\\
&\leq u(\vec x,0)-\frac{ac_*}{2}|\bar\nu|\cos(\langle -e_2,\nu\rangle+\tfrac{\theta_0}2).
\end{aligned}
\]

Writing $\vec z:=\vec y-\vec x\in B_{\delta_\nu}(0)$ and $\eta_\nu:=\frac{a}{4}\cos(\langle -e_2,\nu\rangle+\tfrac{\theta_0}2)$, for some $c_2>0$, we get
\begin{align*}
&u(\vec x+(1+c_2\eta_\nu )\vec z-\nu,0)-u(\vec x,0)\\
&\qquad\qquad \leq u(\vec x+(1+c_2\eta_\nu )\vec z-\nu,0)-u(\vec x+\vec z-\nu,0)-2c_*\eta_\nu |\nu|\\
&\qquad\qquad \leq c_2\eta_\nu 
|\nu| |\nabla u(\cdot,0)|-2c_*\eta_\nu |\nu|.
\end{align*}
Recall that $|\nabla u|\leq Cc_*$, thus if $c_2$ is sufficiently small independent of $c_*$, we get
\[
u(\vec x+(1+c_2\eta_\nu )\vec z-\nu,0)-u(\vec x,0)\leq -c_*\eta_\nu |\nu|,
\]
which implies \eqref{7.12'}.
\end{proof}

We present a different version of the comparison result from Lemma \ref{L.7.6}. Unlike the previous version, which required $r >0$ to be sufficiently small, the following result holds for any $r > 0$ under the additional assumption that $C\eta\eps \geq 2\eps\|\nabla b_1\|_\infty + \|b_1-b_2\|_\infty$. In view of \eqref{e.b}, the proof follows with minor adjustments (taking $\eps_2 = 0$ and $\eps_1 = C\eps\eta$).  Moreover, condition (2) is no longer required since $f \equiv 0$.


\begin{lemma}{\rm [A comparison result]}\lb{L.6.3}
Let $u_1$ and $u_2$ be viscosity solutions to \eqref{2.6} in $ B_1\times [-\tau,\tau]$ with $f\equiv0$, and with ${b}_1$ and ${b}_2$ in place of ${b}$, respectively. Let $r>0$ and $\eps\in (0,1)$.
We assume (1)(3)--(5) and \eqref{72.2} from Lemma \ref{L.7.6} with $\gamma=1$ and that for some $C\geq 1$, 
\beq\lb{add}
C\eta\eps\geq 2\eps\|\nabla b_1\|_\infty+\|b_1-b_2\|_\infty
\eeq
(instead of  assuming $C\eta\geq r$), then the conclusion of the lemma holds the same.
\end{lemma}

Recall \eqref{def.X} and $X(t):=X(t;0,0)$. Let us fix $\theta_*\in (\theta_{*,2},\frac{\pi}{2})$ with $\theta_{*,2}$ from Theorem \ref{T.8.4}. Next, we take $r=r_0\geq 1$ from Lemma \ref{L.6.1} such that  the conclusion of the lemma holds uniformly for all $\theta_0\in [\frac{\pi}{3},\theta_*]$.
Then, we define 
\[
u_r:= \frac1r u(rx,ry-\Theta(0)rt,rt)
\]
which satisfies \eqref{3.17} and \eqref{3.18}, with $\gamma=1$.

\begin{proposition}{\rm [Gain up to the boundary]}\lb{L.5.10}
Under the assumptions of Lemma \ref{L.7.4} and Lemma \ref{L.6.1},  
suppose $(0,0)\in\Gamma_u\cap \calQ_{1,c}$. For $\theta\in (\frac\pi3,\frac{\pi}{2})$, $\mu=-e_2$, $\alpha=1$ and $\nu\in W_{\theta/2,-e_2}$, let $c_0,\tau_1$ from Lemma \ref{L.7.4} and $c_2,\eta_\nu$  from Lemma \ref{L.6.1}. 
Then there exist $C,c_3>0$ such that if 
$
\eta_\nu\geq Cr \|\nabla b\|$,
we have 
\[
\sup_{y\in B_{(1+c_3\eta_\nu )\eps}(x,t)}u_{r}(y-\nu,t) \leq u_{r}(x,t)\quad\text{ in } B_{1/2}\times [-\tau/2,\tau]
\]
where $\eps:=|\nu|\sin(\tfrac{\theta}{2})$.
\end{proposition}

\begin{proof}
It follows from Theorem \ref{T.2.5} that if $\tau_1\leq c/c_*$ for some $c>0$, then 
$\Gamma_{u_r}(t)$ is away from $B_{1/2}({-3e_2/4})$ for $t\in [-\tau_1,\tau_1]$. Since $u_r(x,y)$ is $1/r$-periodic in $x$, by assuming $r\geq 2$ and further taking $\tau_1$ small, we know that $\Gamma_{u_r}(t)\subseteq \{(x,y)\,|\, y\leq -\frac12\}$.  Therefore, it follows from Lemma \ref{L.6.1} that 
there is $c_2\in (0,1)$ independent of $r\geq r_0$ and $c_*$ such that
for any $\nu\in W_{\theta/2,-e_2}$  with $|\nu|\in (0,\frac18)$ and for any $(x,y)$ with $y\leq -\frac12$ we have
\[
\sup_{B_{(1+c_2\eta_\nu )\epsilon}(x,y)}u_r(\cdot-\nu,t)\leq u_r(x,y,t)-c_*\eta_\nu\eps
\]
where $\eta_\nu=\frac{a}{4}\cos(\langle -e_2,\nu\rangle+\tfrac{\theta}2)$ and $\eps=|\nu|\sin(\tfrac{\theta}{2})$. Since $\langle -e_2,\nu\rangle,\frac{\theta}2<\frac\pi2$, $\eta_\nu$ is strictly positive depending only on $\theta$.

We denote $\vec{x}_1:=-\frac{3}{4}e_2$, then for all $\vec{z}:=(x,y)\in B_{1/8}(\vec x_1)$ we have $y\in  (-1,-\frac12)$. Consequently,  $u_r(\vec{z},t)\leq Cc_*$ and so for some $c>0$,
\beq\lb{7.41'}
\sup_{\vec{y}\in B_{(1+c_2\eta_\nu ) \eps}(\vec{z})}u_{r}(\vec y- \nu,t)\leq (1-c\eta_\nu\eps) u_{r}(\vec z,t)\quad\text{ in } B_{1/8}(\vec x_1)\times (-\tau_1,\tau_1).
\eeq

Now, we let $\varphi_\eta$ from Lemma \ref{L.7.5} with $\tau:=\tau_1$ and $\eta:=c\eta_\nu$. Define
\[
u_1(\vec z,t):=u_{r}(\vec z-\nu,t)\quad\text{and}\quad u_2(\vec z,t):=u_{r}(\vec z,t),
\]
and let $ b_1(\vec z,t)=b_{r }(\vec z-\nu,t),  b_2=b_{r }$, where $b_r:=b_{r,1}$ is defined in \eqref{3.18}.
Then, 
$u_1$ and $u_2$ are viscosity solutions to \eqref{1.1} with $f\equiv 0$ and with ${b}_1$ and ${b}_2$ in place of ${b}$, respectively. Direct computation yields in $B_1\times [-\tau_1,\tau_1]$,
\[
2\eps\|\nabla b_1\|_\infty+\|b_1-b_2\|_\infty\leq Cr\eps \|\nabla b\|.
\]
By assuming $\|\nabla b\|_\infty$ to be sufficiently small, we get \eqref{add}. 
Moreover, $u_{r}(\cdot,t)$ is non-decreasing along all directions of $W_{\theta,-e_2}$.
And, by \eqref{7.41'} and by assuming $\tau_1\leq c_2$, we verified the conditions in Lemma \ref{L.6.3}.

Since $u_{r}(\cdot,t)$ is non-decreasing along all directions of $W_{\theta,-e_2}$, the definition of $\eps$ yields
\[
\sup_{\vec y\in B_{\eps }(\vec z)}u_1(\vec y,t)\leq u_2(\vec z,t).
\]
For $(\vec z,t)\in B_{1/8}(x_1)\times (-\tau_1,\tau_1)$, $\tau_1\leq c_2$ and \eqref{7.41'} yield
\begin{align*}
\sup_{\vec y\in B_{(1+\tau_1\eta)\eps}(\vec z,t)}u_1(\vec y,t)
&\leq \sup_{\vec y\in B_{(1+c_2\eta_\nu)\eps}(\vec z,t)}u_{r}(\vec y-\nu,t)\\
&\leq  (1-c\eta_\nu\eps)u_{r}(\vec y-\nu,t)\leq (1-C\eta\eps )u_2(\vec z,t).    
\end{align*}
Hence, the conditions of Lemma \ref{L.6.3} all hold and it implies that
\[
\sup_{y\in B_{(1+\tau k\eta)\eps}(x,t)}u_1(y,t) \leq u_2(x,t)\quad\text{ in } B_{1/2}\times [-\tau/2,\tau],
\]
yielding the conclusion.
\end{proof}

As a consequence of the lemma, we obtained the improvement of the monotonicity of $u$ from non-decreasing along all directions in $W_{\theta,-e_2}$ to all directions in $W_{\theta+\kappa,-e_2}$ for some $\kappa>0$. Moreover, if $\theta\in [\theta_0,\theta_*]$ for some $\theta_0>\frac{\pi}{3}$ and $\theta_*\in (\theta_{*,2},\frac{\pi}{2})$, then the constant $\kappa>0$ is uniform for all such $\theta$. 

With this gain of monotonicity and with the help of Theorem \ref{T.8.4}, we are ready to show the second main result of the paper that the free boundary for the 2D advection Muskat problem is uniformly $C^1$ after a finite time.

\begin{proof}[Proof of Theorem \ref{T.main.2}]
The existence of solutions is proved in Theorem \ref{T.2.5}.

Due to the assumption that $\calG(F_0)(x)<1$, we have
\[
m_1=\min_{x\in \bbT}\{-c_*\calG(F_0)(x) - \Theta(x)\} + c_*\geq cc_*-\max_{x\in\bbT}\Theta(x)
\]
for some $c>0$. Hence, by taking $c_*$ large, we have $M_1 \asymp c_*$ and \eqref{detla1} holds.
By Theorem \ref{T.2.5}, the solution $F(\cdot,\cdot)$ is locally uniformly Lipschitz continuous in space and time. Since $F_0$ is Lipschitz continuous with Lipschitz constant $<1/\sqrt{3}$ and $\|\Theta\|_{\Lip}$ is small, Theorem \ref{T.2.5} yields that we can find $T\geq 1$ such that $F(\cdot,t)$ is Lipschitz continuous with Lipschitz constant $\cot\theta_0<1/\sqrt{3}$  for some $\theta_0>\frac{\pi}{3}$ and for all $t\in [0,T]$. Moreover, $T$ can be taken to be arbitrarily large if we further assume $\|\Theta\|_{\Lip}$ to be small.

Let $u(x,y,t):=(\phi^t(x,y)-c_* y)_+$ with $\phi^t$ from \eqref{a.1.1}. By Lemma \ref{L.5.5}, $u$ is monotone non-decreasing along all directions of $W_{\theta_0,-e_2}$ for $t\in [0,T]$.
By the condition, Lemma  \ref{C.6.4} yields that $u$ is non-degenerate. Since $|F_t|$ is uniformly finite by Theorem \ref{T.2.5}, the condition \eqref{bcond} holds.  We use $r$ from the paragraph before Proposition \ref{L.5.10}, and $\tau$ from the proposition, and we assume that $T\geq r\tau$. Thus, we can apply Proposition \ref{L.5.10} to obtain that, if $\|\Theta\|_{\Lip}$ is sufficiently small, for all $t\in [r\tau/2,T]$, $F(\cdot,t)$ is Lipschitz continuous with Lipschitz constant $c_{L,1}$ where
\[
c_{L,1}\leq \cot(\theta_1)\quad\text{and}\quad \theta_1:=\frac{\pi}{3}+\kappa
\]
and $\kappa>0$ only depends on $c_3$ from the lemma. Without loss of generality, we can assume that $\theta_1\leq \theta_{*,2}<\frac{\pi}{2}$ with $\theta_{*,2}$ from Theorem \ref{T.8.4}, and $\theta_0<\theta_1$.

Then, setting $T_1:=r\tau/2$, we consider $\widetilde F(\cdot,t)= F(\cdot, t+T_1)$ and repeat the above argument. Theorem \ref{T.2.5} yields that $\widetilde F(\cdot, t)$ is Lipschitz continuous with Lipschitz constant $\cot\theta_0$ for some $\theta_0>\frac{\pi}{3}$ and for $t\in [0,T]$ (with the same $T$ if $\|\Theta\|_{\Lip}$ is small). It follows from Proposition \ref{L.5.10} that $\widetilde F(\cdot,t)$ for all $t\in [r\tau/2,T]$ is Lipschitz continuous with Lipschitz constant $c_{L,1}$. Since $T\geq r\tau$, by iteration, we get that $F(\cdot,t)$ is uniformly Lipschitz continuous    with Lipschitz constant $c_{L,1}$ for all $t\geq T_1$.

Now, we apply Proposition \ref{L.5.10} again for $t\in [T_1,\infty)$ with the knowledge that $u(\cdot,t)$ monotone non-decreasing along all directions of $W_{\theta_1,-e_2}$. Then we get that $F(\cdot,t)$ is uniformly Lipschitz continuous    with Lipschitz constant $c_{L,2}$ for all $t\geq T_2$, where
\[
T_2=r\tau,\quad c_{L,2}\leq \cot(\theta_2)\quad\text{and}\quad \theta_2:=\frac{\pi}{3}+\kappa+\kappa'
\]
for some $\kappa'>0$. Without loss of generality, we can assume $\kappa'<\pi/2-\theta_{*,2}$.
By iteration, we get that
that $F(\cdot,t)$ is uniformly Lipschitz continuous    with Lipschitz constant $c_{L,N}$ for all $t\geq T_N$, where
\[
T_N=Nr\tau/2,\quad c_{L,N}\leq \cot(\theta_N)\quad\text{and}\quad \theta_N:=\frac{\pi}{3}+\kappa+(N-1)\kappa'
\]
and $\theta_N\in [\theta_{*,2},\theta_{*,2}+\kappa')$. The constants $r,\tau,\kappa'$ and the smallness condition on $\|\Theta\|_{\Lip}$ are the same in the iteration process, which is because we have $\theta_i\in  (\theta_1,\theta_{*,2})$ for all $i=2,3,\ldots,N-1$.

Finally, since $u$ is non-decreasing along all directions in $W_{\theta_{*,2},-e_2}$ when $t\geq T_N$, we apply Theorem \ref{T.8.4} to get that $F(\cdot,t)$ is uniformly $C^1$ for all $t\geq T_N+r\tau$.
\end{proof}



\begin{thebibliography}{10}

\bibitem{ala21}
T.~Alazard.
\newblock Convexity and the hele--shaw equation.
\newblock {\em Water Waves}, 3(1):5--23, 2021.

\bibitem{ak23}
T.~Alazard and H.~Koch.
\newblock The {Hele-Shaw} semi-flow.
\newblock {\em arXiv preprint arXiv:2312.13678}, 2023.

\bibitem{ams20}
T.~Alazard, N.~Meunier, and D.~Smets.
\newblock Lyapunov functions, identities and the cauchy problem for the
  hele--shaw equation.
\newblock {\em Communications in Mathematical Physics}, 377(2):1421--1459,
  2020.

\bibitem{blank2001sharp}
I.~Blank.
\newblock Sharp results for the regularity and stability of the free boundary
  in the obstacle problem.
\newblock {\em Indiana University Mathematics Journal}, pages 1077--1112, 2001.

\bibitem{Caf87}
L.~A. Caffarelli.
\newblock A harnack inequality approach to the regularity of free boundaries.
  {Part I}: Lipschitz free boundaries are {$C^{1,\alpha}$}.
\newblock {\em Revista Matem{\'a}tica Iberoamericana}, 3(2):139--162, 1987.

\bibitem{cbook}
L.~A. Caffarelli and S.~Salsa.
\newblock {\em A geometric approach to free boundary problems}, volume~68.
\newblock American Mathematical Soc., 2005.

\bibitem{cameron2018global}
S.~Cameron.
\newblock Global well-posedness for the two-dimensional {Muskat} problem with
  slope less than 1.
\newblock {\em Analysis \& PDE}, 12(4):997--1022, 2018.

\bibitem{cameron2020global}
S.~Cameron.
\newblock Global wellposedness for the 3d {Muskat} problem with medium size
  slope.
\newblock {\em arXiv preprint arXiv:2002.00508}, 2020.

\bibitem{ccfg2013}
{\'A}.~Castro, D.~C{\'o}rdoba, C.~Fefferman, and F.~Gancedo.
\newblock Breakdown of smoothness for the {Muskat} problem.
\newblock {\em Archive for Rational Mechanics and Analysis}, 208(3):805--909,
  2013.

\bibitem{ccfg2016}
A.~Castro, D.~C{\'o}rdoba, C.~Fefferman, and F.~Gancedo.
\newblock Splash singularities for the one-phase {Muskat} problem in stable
  regimes.
\newblock {\em Archive for Rational Mechanics and Analysis}, 222(1):213--243,
  2016.

\bibitem{cheng2016well}
C.~A. Cheng, R.~Granero-Belinch{\'o}n, and S.~Shkoller.
\newblock Well-posedness of the {Muskat} problem with {$H^2$} initial data.
\newblock {\em Advances in Mathematics}, 286:32--104, 2016.

\bibitem{CJK}
S.~Choi, D.~Jerison, and I.~Kim.
\newblock Regularity for the one-phase {{Hele-Shaw}} problem from a lipschitz
  initial surface.
\newblock {\em American journal of mathematics}, 129(2):527--582, 2007.

\bibitem{CJK2}
S.~Choi, D.~Jerison, and I.~Kim.
\newblock Local regularization of the one-phase {Hele-Shaw} flow.
\newblock {\em Indiana University mathematics journal}, pages 2765--2804, 2009.

\bibitem{constantin2016muskat}
P.~Constantin, D.~C{\'o}rdoba, F.~Gancedo, L.~Rodr{\'\i}guez-Piazza, and R.~M.
  Strain.
\newblock On the {Muskat} problem: global in time results in 2d and 3d.
\newblock {\em American Journal of Mathematics}, 138(6):1455--1494, 2016.

\bibitem{constantin2013global}
P.~Constantin, D.~C{\'o}rdoba, F.~Gancedo, and R.~M. Strain.
\newblock On the global existence for the {Muskat} problem.
\newblock {\em Journal of the European Mathematical Society (EMS Publishing)},
  15(1), 2013.

\bibitem{constantin1993global}
P.~Constantin and M.~Pugh.
\newblock Global solutions for small data to the {Hele-Shaw} problem.
\newblock {\em Nonlinearity}, 6(3):393--415, 1993.

\bibitem{cordoba2011interface}
A.~C{\'o}rdoba, D.~C{\'o}rdoba, and F.~Gancedo.
\newblock Interface evolution: the {Hele-Shaw} and {Muskat} problems.
\newblock {\em Annals of mathematics}, pages 477--542, 2011.

\bibitem{CKY}
K.~Craig, I.~Kim, and Y.~Yao.
\newblock Congested aggregation via newtonian interaction.
\newblock {\em Archive for Rational Mechanics and Analysis}, 227(1):1--67,
  2018.

\bibitem{dah}
B.~Dahlberg.
\newblock Harmonic functions in lipschitz domains.
\newblock {\em Harmonic analysis in Euclidean spaces}, pages 313--322, 1979.

\bibitem{david2021free}
N.~David and B.~Perthame.
\newblock Free boundary limit of a tumor growth model with nutrient.
\newblock {\em Journal de Math{\'e}matiques Pures et Appliqu{\'e}es},
  155:62--82, 2021.

\bibitem{David_S}
N.~David and M.~Schmidtchen.
\newblock On the incompressible limit for a tumour growth model incorporating
  convective effects.
\newblock {\em Communications on Pure and Applied Mathematics},
  77(5):2613--2650, 2024.

\bibitem{de2017paradifferential}
T.~De~Poyferr{\'e} and Q.-H. Nguyen.
\newblock A paradifferential reduction for the gravity-capillary waves system
  at low regularity and applications.
\newblock {\em Bulletin de la Soci{\'e}t{\'e} Math{\'e}matique de France},
  145(4):643--710, 2017.

\bibitem{dong21}
H.~Dong, F.~Gancedo, and H.~Q. Nguyen.
\newblock Global well-posedness for the one-phase {Muskat} problem.
\newblock {\em Communications on Pure and Applied Mathematics},
  76(12):3912--3967, 2023.

\bibitem{dong23}
H.~Dong, F.~Gancedo, and H.~Q. Nguyen.
\newblock Global well-posedness for the one-phase {Muskat} problem in 3d.
\newblock {\em arXiv preprint arXiv:2308.14230}, 2023.

\bibitem{EJ}
C.~M. Elliott and V.~Janovsk{\`y}.
\newblock A variational inequality approach to {{Hele-Shaw}} flow with a moving
  boundary.
\newblock {\em Proceedings of the Royal Society of Edinburgh Section A:
  Mathematics}, 88(1-2):93--107, 1981.

\bibitem{escher1997classical}
J.~Escher and G.~Simonett.
\newblock Classical solutions of multidimensional {Hele--Shaw} models.
\newblock {\em SIAM Journal on Mathematical Analysis}, 28(5):1028--1047, 1997.

\bibitem{evans}
L.~C. Evans.
\newblock {\em Partial Differential Equations}, volume~19.
\newblock American Mathematical Soc., 2010.

\bibitem{figalli20}
A.~Figalli, X.~Ros-Oton, and J.~Serra.
\newblock Generic regularity of free boundaries for the obstacle problem.
\newblock {\em Publications math{\'e}matiques de l'IH{\'E}S}, 132(1):181--292,
  2020.

\bibitem{gancedo2019muskat}
F.~Gancedo, E.~Garcia-Juarez, N.~Patel, and R.~M. Strain.
\newblock On the {Muskat} problem with viscosity jump: global in time results.
\newblock {\em Advances in Mathematics}, 345:552--597, 2019.

\bibitem{jacobs2022tumor}
M.~Jacobs, I.~Kim, and J.~Tong.
\newblock Tumor growth with nutrients: regularity and stability.
\newblock {\em Communications of the American Mathematical Society},
  3(04):166--208, 2023.

\bibitem{JK1}
D.~Jerison and I.~Kim.
\newblock The one-phase {{Hele-Shaw}} problem with singularities.
\newblock {\em The Journal of Geometric Analysis}, 15(4):641--667, 2005.

\bibitem{kim2003}
I.~Kim.
\newblock Uniqueness and existence results on the {{Hele-Shaw}} and the stefan
  problems.
\newblock {\em Archive for Rational Mechanics \& Analysis}, 168(4), 2003.

\bibitem{kim3}
I.~Kim.
\newblock Long time regularity of solutions of the {Hele--Shaw} problem.
\newblock {\em Nonlinear Analysis: Theory, Methods \& Applications},
  64(12):2817--2831, 2006.

\bibitem{kim2}
I.~Kim.
\newblock Regularity of the free boundary for the one phase {Hele--Shaw}
  problem.
\newblock {\em Journal of Differential Equations}, 223(1):161--184, 2006.

\bibitem{contact}
I.~Kim.
\newblock Homogenization of a model problem on contact angle dynamics.
\newblock {\em Communications in Partial Differential Equations},
  33(7):1235--1271, 2008.

\bibitem{KimMellet2009}
I.~Kim and A.~Mellet.
\newblock Homogenization of a {H}ele-{S}haw problem in periodic and random
  media.
\newblock {\em Arch. Ration. Mech. Anal.}, 194(2):507--530, 2009.

\bibitem{kim2023incompressible}
I.~Kim and A.~Mellet.
\newblock Incompressible limit of a porous media equation with bistable and
  monostable reaction term.
\newblock {\em SIAM Journal on Mathematical Analysis}, 55(5):5318--5344, 2023.

\bibitem{kimzhang2024}
I.~Kim and Y.~P. Zhang.
\newblock Regularity of {{Hele-Shaw}} flow with source and drift.
\newblock {\em Annals of PDE}, 10(2):20, 2024.

\bibitem{persis}
J.~King, A.~Lacey, and J.~Vazquez.
\newblock Persistence of corners in free boundaries in {{Hele-Shaw}} flow.
\newblock {\em European Journal of Applied Mathematics}, 6(05):455--490, 1995.

\bibitem{maury2010}
B.~Maury, A.~Roudneff-Chupin, and F.~Santambrogio.
\newblock A macroscopic crowd motion model of gradient flow type.
\newblock {\em Mathematical Models and Methods in Applied Sciences},
  20(10):1787--1821, 2010.

\bibitem{muskat}
M.~{Muskat}.
\newblock Two fluid systems in porous media. the encroachment of water into an
  oil sand.
\newblock {\em Physics}, 5(9):250--264, 1934.

\bibitem{nguyen2022global}
H.~Q. Nguyen.
\newblock Global solutions for the {Muskat} problem in the scaling invariant
  {Besov space $\dot B^1_{\infty, 1}$}.
\newblock {\em Advances in Mathematics}, 394:108122, 2022.

\bibitem{np20}
H.~Q. Nguyen and B.~Pausader.
\newblock A paradifferential approach for well-posedness of the {Muskat}
  problem.
\newblock {\em Archive for Rational Mechanics and Analysis}, 237(1):35--100,
  2020.

\bibitem{PQV}
B.~Perthame, F.~Quir{\'o}s, and J.~L. V{\'a}zquez.
\newblock The {Hele--Shaw} asymptotics for mechanical models of tumor growth.
\newblock {\em Archive for Rational Mechanics and Analysis}, 212(1):93--127,
  2014.

\bibitem{povzar2015homogenization}
N.~Po{\v{z}}{\'a}r.
\newblock Homogenization of the {{Hele-Shaw}} problem in periodic
  spatiotemporal media.
\newblock {\em Archive for Rational Mechanics and Analysis}, 217(1):155--230,
  2015.

\bibitem{richardson1972hele}
S.~Richardson.
\newblock Hele shaw flows with a free boundary produced by the injection of
  fluid into a narrow channel.
\newblock {\em Journal of Fluid Mechanics}, 56(4):609--618, 1972.

\bibitem{schwab2024well}
R.~Schwab, S.~Tu, and O.~Turanova.
\newblock Well-posedness for viscosity solutions of the one-phase {Muskat}
  problem in all dimensions.
\newblock {\em arXiv preprint arXiv:2404.10972}, 2024.

\bibitem{HS1898}
H.~S.~H. Shaw.
\newblock Investigation of the nature of surface resistance of water and of
  stream-line motion under certain experimental conditions.
\newblock In {\em Inst. NA.}, 1898.

\bibitem{SulakTuranova}
A.~Sulak and O.~Turanova.
\newblock The incompressible limit of an inhomogeneous model of tissue growth.
\newblock {\em arXiv:2503.19849}, 2025.

\bibitem{tong2025convergence}
J.~Tong and Y.~P. Zhang.
\newblock Convergence of free boundaries in the incompressible limit of tumor
  growth models.
\newblock {\em Journal de Math{\'e}matiques Pures et Appliqu{\'e}es},
  203:103752, 2025.

\bibitem{turanova2025hele}
O.~Turanova and Y.~P. Zhang.
\newblock A {{Hele-Shaw}} problem with interior and free boundary oscillation:
  well-posedness and homogenization.
\newblock {\em arXiv preprint arXiv:2508.13441}, 2025.

\end{thebibliography}




\end{document}